\documentclass{amsart}

\usepackage{times}
\usepackage[margin=1in]{geometry}
\usepackage{enumerate}
\usepackage{amssymb}
\usepackage{latexsym}
\usepackage{longtable}
\usepackage{graphics}
\usepackage{comment}
\usepackage{amsmath} 
\usepackage{multirow}
\copyrightinfo{2025}{}
\allowdisplaybreaks[1]

\newcommand{\Dif}{\mathrm{Dif}}
\newcommand{\SubF}{\mathrm{SubF}}
\newcommand{\dl}{\mathrm{dl}}
\newcommand{\Fac}{\mathrm{Fac}}

\theoremstyle{definition}
\newtheorem{definition}{Definition}[section]

\theoremstyle{plain}
\newtheorem{lemma}[definition]{Lemma}
\newtheorem{theorem}[definition]{Theorem}
\newtheorem{proposition}[definition]{Proposition}
\newtheorem{corollary}[definition]{Corollary}
\newtheorem{example}[definition]{Example}
\newtheorem{remark}[definition]{Remark}
\newtheorem{question}[definition]{Question}
\newtheorem{problem}[definition]{Problem}

\begin{document}

\title[Subindices and subfactors of infinite groups]
{Subindices and subfactors of infinite groups and numbers}

\author{M.H. Hooshmand}
\address{Department of Mathematics, Shi.C., Islamic Azad University, Shiraz, Iran}
\email{MH.Hooshmand@iau.ac.ir}
\email{hadi.hooshmand@gmail.com}

\maketitle

\begin{abstract}
The theory of subfactors of groups, together with the associated notions of subindices and index stability for groups and their subsets, has recently been introduced and systematically developed. These concepts exhibit deep connections with additive combinatorics and number theory, relating to important topics such as packing and covering numbers, syndetic sets, group diameters, special integer sequences (e.g., primes and Fibonacci numbers), and classical rational sequences (e.g., Bernoulli numbers). Following the initial paper presented in 2020, two subsequent works further investigated these ideas within the framework of finite groups.
In the present paper, in addition to advancing several aspects of the topic, we focus on infinite groups, with particular emphasis on groups of numbers. In this context, we introduce the RSFA (Right Subfactor Algorithm) for infinite groups and resolve several previously open problems. One of the important results is that every infinite group is index-unstable.
We also correct several earlier inaccuracies and establish a weak version of a conjecture concerning differences of prime numbers.
 Furthermore, we determine the exact subindices for several notable sequences of integers and provide a general criterion for index stability and non-index stability of subsets in countable groups. Finally, we investigate the index stability of infinite groups and present a collection of related projects, problems, questions, and conjectures.
\end{abstract}

\maketitle

\enlargethispage*{4mm}
\thispagestyle{empty}

\section{Introduction and preliminaries} \label{Introduction}
The study of factorizations of groups offers rich mathematical landscapes intertwined with questions. In some recent studies, a challenging open question concerning the factorization of finite groups by subsets, earlier in 2014 \cite{OverflowFactorSubsets}, and then in the Kourovka Notebook (Problem 19.35, 20.37),  is stated. The question has only been partially answered in existing literature (e.g., see \cite{KourovkaNotebook, Hooshmand2021}). Motivated by inquiries into the mentioned topic, Upper Periodic Subsets,  and functional equations on groups and semigroups, we were led to introduce subfactors and the generalization of subgroup indices to arbitrary subsets of groups \cite{Hooshmand2020}. The theory of subfactors and subindices of groups has been further developed in subsequent papers, such as \cite{Hooshmand2021, Hooshmand2023}. Also, in \cite{Kabenyuk}, a connection between the topic and an interesting concept from graph theory has been introduced.

One of the central constructs we employ is that of the \emph{direct product} of two subsets.
 For arbitrary subsets \( A, B \subseteq G \), the product \( AB \) is said to be \emph{direct}, denoted \( A \cdot B \), if each element of \( AB \) has a unique representation \( ab \) with \( a \in A \) and \( b \in B \).
For nonempty subsets $A,B$, we have \(AB=A\cdot B\) if and only if \(A^{-1}A\cap BB^{-1}=\{1\}\), and if and only if \(aB\cap a'B=\emptyset\) for all distinct elements  \(a,a'\in A\)
 (see \cite{Hooshmand2020, Hooshmand2021,Hooshmand2023} for more details).
	Note that \(G=A \cdot B\) if and only if \(G=A B\) and the product \(A B\) is direct
	(the additive notation is \(G=A\dot{+} B\)).

	If \(G=A \cdot B\) , then \(A\) is a left  factor of \(G\) related to \(B\),
and call  \(B\) a \emph{right factor complement} of \(A\).
	We call \(A\) a left factor of \(G\) if and only if \(G=A \cdot B\) for some \(B\subseteq G\).
	For example, every subgroup is a left (resp. right) factor
	related to its right (resp. left) transversal, hence it is a two-sided factor of \(G\).
	In reference \cite{Hooshmand2020}, the author achieved a generalization of factors
	that not only do not have the deficiency of factors but also lead to the important concept of subindices for all subsets of groups.\\
	Let \(A\) be a fixed subset of \(G\). We call \(B\)  a right subfactors of \(G\)
	related to \(A\)  if \(B\) is an inclusion-maximal subset of \(G\) with respect to the property \(AB=A\cdot B\).
Also, we say \(B\) is a right subfactor of \(G\) if it is a right subfactor related to some subset
	of \(G\) (left subfactors are defined analogously). \\
	It is proved that every subset of a group has related right and left subfactors.
	Also, \(B\) is a right subfactor (of \(G\)) related to \(A\) if and only if
	\begin{equation} \label{SubfactorCondition}
		A^{-1}A\cap BB^{-1}=\{ 1\}\; , \; A^{-1}AB=G.
	\end{equation}
To analyze the behavior  of factors and subfactors, we also define left and right \emph{difference sets}, \( \Dif_\ell(A) = A^{-1}A \) and \( \Dif_r(A) = AA^{-1} \). A subset \( A \) is said to be a left (resp. right) \emph{difference-generating set} if \( \Dif_\ell(A) = G \) (resp. \( \Dif_r(A) = G \)).
For convenience, in this paper, we often use ``\(\Dif\)'' instead of ``\(\Dif_\ell\)''.  Denoting the \(n\)-iteration of \(\Dif\)  by \(\Dif^n\) we have the collection of nested subsets
\(\{\Dif^n(A)\}_{n=1}^\infty\), and define
\[
\Dif^\infty(A)=\lim_{n\rightarrow \infty} \Dif^n(A):=\bigcup_{n=1}^\infty \Dif^n(A),
\]
namely left \(\infty\)-difference of \(A\).
\cite[Lemma 2.1]{Hooshmand2020} proves that for every non-empty subset \(A\) the left \(\infty\)-difference of \(A\) is a subgroup of \(G\).
Also, we have \(\Dif^n(A)=\Dif^n(gA)\) for all \(1\leq n\leq\infty\), \(g\in G\), and
\begin{equation}
\Dif^\infty(A)=\Dif^\infty(a^{-1}A)=\Dif^\infty(A^{-1}A)=\langle a^{-1}A\rangle=\langle A^{-1}A\rangle=\Dif^\infty(\Dif(A))=\Dif(\Dif^\infty(A))
 \; ; \; \forall a\in A.
\end{equation}
Note that \(\Dif^0(A):=A\) and we don't have \(\Dif^0(A)\subseteq \Dif^1(A)\) in general, but
if \(A\subseteq A^{-1}A\) (e.g., if \(a_0A\subseteq A\) for some \(a_0\in A\) which means \(A\) is left upper \(a_0\)-periodic, see
\cite{Hooshmand2011}), then \(\{\Dif^n(A)\}_{n=0}^\infty\) is also a collection of nested subsets
 and \(\Dif^\infty(A)=\langle A\rangle=\langle \Dif(A) \rangle\). It is interesting to know that
for every arbitrary subset \(A\)
 \[
\langle A\rangle =\langle A\cup\{1\}\rangle=\Dif^\infty(A\cup\{1\})
\]
which gives another definition for the subgroup generated by \(A\).
More information about this can be seen in \cite{Hooshmand2020, Hooshmand2023}.\\
For every $A\subseteq G$, either there exists \underline{the least} non-negative integer
\(N\) such that \(\Dif^N(A)= \Dif^{N+1}(A)\) (equivalently \(\Dif^N(A)\leq G\) if $A\neq \emptyset$)
or the sequence \(\{\Dif^n(A)\}_{n=1}^\infty\) is
strictly increasing. In the first case, we
define \(\dl^\infty (A)=\dl_\ell^\infty (A):=N\) and we have \(\Dif^\infty(A)=\Dif^{N}(A)\), also we put \(\dl^\infty (A):=\infty\)
for the second case. We call \(\dl^\infty (A)\) left \(\infty\)-difference length of \(A\). Therefore \(N=\dl^\infty (A)\) is the least
number (including infinity) satisfying \(\Dif^\infty(A)=\Dif^{N}(A)\).\\
Let \(A\subseteq H\leq G\), and put \(A^1:=A\cup\{1\}\)
(\(A^0:=A\cup\{0\}\) for additive notation).
By \(H=\langle A \rangle^{(N)}\)
we mean \(H=\langle A\rangle\) (equivalently \(H=\langle A^{1}\rangle\)) and \(dl^\infty (A^1)=N\), and say
\(A^1\) is a generating set of \(H\) with left \(\infty\)-difference length \(N\).
For example, in the additive group of integers and \(\mathbb{Z}_6\)  we have
\[
\mathbb{Z}=\langle \mathbb{N} \rangle^{(1)}=\langle 2,3 \rangle^{(\infty)}\;
, \; \mathbb{Z}_6=\langle 0,1,5 \rangle^{(2)}=\langle 2,3 \rangle^{(3)}=\langle 1 \rangle^{(3)}.
\]
For finite groups, if \(|A|>\frac{|G|}{2}\), then \(dl^\infty (A)=1\)
which means $A$ is a left deference generating set for $G$ (see Theorem~\ref{BasicTheorem}), and so \(G=\langle A \rangle^{(1)}\).
As an example of infinite deference length, every finite subset of integers
with more than one element has the \(\infty\)-difference length of infinity.
One can see more information for this in \cite{Hooshmand2020}
that one of them is: for every set \(A\) of integers,  \(\mathbb{Z}=\langle A \rangle^{(\infty)}\) if and only if
\begin{enumerate}[(1)]
\item For every integer \(n\) there exist a natural number \(\kappa=2^q\) and  \(a_1,\cdots,a_\kappa\in A\) such that
\[
n=a_1+\cdots+a_{\frac{\kappa}{2}}-(a_{\frac{\kappa}{2}+1}+\cdots+a_\kappa);
\]
\item Denoting by \(\kappa(n)\) the least \(\kappa\) obtained from (1),  the set of all \(\kappa(n)\)
, where \(n\) runs over all integers, is unbounded above.
\end{enumerate}
\begin{example}\label{WaringExample}
Consider the additive group of integers,  and put $\mathbb{Z}_e=2\mathbb{Z}$, $\mathbb{Z}_o=2\mathbb{Z}+1$, and
\[
A=\{m^2:m\in \mathbb{Z}^+\cup\{0\}\}.
\]
 Then \(dl^\infty (A)=2\), since  \(\Dif(A)=\mathbb{Z}_o\cup 4\mathbb{Z}\), thus \(\Dif^2(A)=\Dif^3(A)=\Dif^\infty(A)=\mathbb{Z}\)
   and so  $\mathbb{Z}=\langle A\rangle^{(2)}$ (we arrive at a related open problem in section 3).\\
Now, put
\(
A'=\{2^m:m\in \mathbb{Z}^+\}\cup\{0\}.
\)
We claim that \(dl^\infty (A')=\infty\)
and so \(\mathbb{Z}_e=\langle A' \rangle^{(\infty)}\).

 For proving it we use the next argument from \cite{stackwaring2022}. Note that \(A'\) has exactly \(m\) positive elements smaller than or equal to \(2^m\). One can form \(m^2\) pairs of them, so \(\Dif(A')\) has at most \(m^2\)  elements with an absolute value smaller than or equal to \(2^m\) if $m>2$ (the larger elements don't play a role, which can be seen by looking at the binary expansion). That means \(\Dif^n(A')\) has at most \(m^{2^n}\) elements with an absolute value smaller than or equal to \(2^m\).

Every positive even number \(s\) can be written (uniquely) as the sum of powers of 2 (again, look at the binary expansion): \(s = \sum_{i=1}^{t}{2^{s_i}}\) with \(t\) and each \(s_i\) positive integers. Now \(-2^m \in \Dif(A')\) for every positive integer \(m\), and we can write \(s = 2^{s_t} - (-2^{s_{t-1}}) - (-2^{s_{t-2}}) ... - (-2^{s_1})\) so \(s \in \Dif^t(A')\). A similar sum/difference works for negative even numbers. So if \(dl^\infty (A')=N<\infty\), then \(\Dif^N(A')\) should contain all even numbers.

Now take \(m = N^N\), then \(\Dif^N(A')\) has \(N^{N2^N}=2^{N2^N\log_2 N}\) elements with an absolute value smaller than or equal to \(2^{N^N}\), while there are \(2^{N^N}+1\) even numbers, which is (much) more.
Therefore, such an \(N\) does not exist and so \(dl^\infty (A')=\infty\).
\end{example}
\subsection{Relations of the infinite difference lengthes to the diameter of groups respect to subsets}
The \emph{diameter} of the Cayley graph of $G$ with respect to a symmetric generating set $S$ (or $S \cup S^{-1}$ when $S$ is not symmetric)
 measures the minimal number of steps required to reach every element of $G$ from the identity via successive products of elements in $S$.
Let \( G \) be an arbitrary group and \( S \subseteq G \) a symmetric generating set (i.e.\ \( S = S^{-1} \) and \( 1 \in S \)). Then the \emph{diameter of \( G \) with respect to \( S \)}, denoted \( \mathrm{diam}(G, S) \), is defined as:
\[
\mathrm{diam}(G, S) := \min \left\{ n \in \mathbb{Z}_+ \cup \{\infty\} \,\middle|\, S^n = G \right\},
\]
where \( S^n \) denotes the set of all products of \( n \) elements from \( S \).

Hence
\[
\mathrm{diam}(G, S) = N\in\mathbb{Z}_+  \quad \Leftrightarrow \quad S^N = G \text{ and } S^{N-1} \ne G.
\]

More information can be found in many references such as \cite{BabaiSeress1988, Halasi2021}.
In studying the additive structure of groups, another related concept arise that quantify how a given subset generates or covers the entire group. The \emph{covering number} (or additive basis order) of a subset $A$ of an abelian group $G$  is defined as the smallest integer $n$ such that $nA$ (the $n$-fold sumset) equals $G$. This is a classical notion in additive combinatorics (e.g., see \cite{Additive Number Theory}).  When $A$ is a generating set, the diameter of the Cayley graph for abelian groups coincides with the covering number, as both concepts essentially measure the same generative capacity of $A$.  Understanding the interplay between these notions provides insights into both the growth of sumsets and the combinatorial geometry of groups.\\
The identity $Dif^n(A)=Dif(A)^{2^{n-1}}$, for all $n\geq 1$, implies that
\begin{equation}\label{diam}
\dl^\infty (A)=\left\lceil \log_2 \left( \mathrm{diam}\left( \langle A^{-1}A \rangle, A^{-1}A \right) \right) \right\rceil + 1
\end{equation}
 since $S=A^{-1}A$ is symmetric
and a generating set of $G=\langle A^{-1}A\rangle=\Dif^\infty(A^{-1}A)$.
Note that this formula also considers the case $\left\lceil \log_2 \infty \right\rceil + 1=\infty$
that says $dl^\infty (A)$ is infinite if and only if $\mathrm{diam}\left( \langle A^{-1}A \rangle, A^{-1}A \right)$ is
infinite.  This fact helps us to determine the infinite difference lengthes
of subsets by using the diameter conception topics and also the additive basis order.
For example, consider $A = \{0,1\}$ in $\mathbb{Z}_{10}$, we have: $S = A-A = \{0, 1, 9\}$, $\langle S \rangle = \mathbb{Z}_{10}$,
and $\mathrm{diam}(\mathbb{Z}_{10}, S) = 5$. Substituting into the formula \ref{diam}:
\[
\dl^\infty(A) = \left\lceil \log_2(5) \right\rceil + 1 = \lceil 2.3219\ldots \rceil + 1 = 3 + 1 = 4.
\]
Therefore, $\mathbb{Z}_{10}=\langle 0,1\rangle^{(4)}$.
\subsection{Subfactors and subindices of group subsets} For every $A\subseteq G$, put
\[
\Fac_r(A) = \Fac_r(G:A) := \{\, B \subseteq G : B \; \mbox{is a right factor of \(G\) related to \(A\)} \,\},
\]
\[
\SubF_r(A) = \SubF_r(G:A) := \{\, B \subseteq G : B \; \mbox{is a right subfactor of \(G\) related to \(A\)} \,\},
\]
we have \(\Fac_r(G:A) \subseteq \SubF_r(G:A) \neq \emptyset\).
But \(\Fac_r(G:A) \neq \emptyset\) if and only if \(A\) is a left factor of \(G\).
Very often, in our topic, $\SubF_r(G:A)$ can be replaced by $\SubF^1_r(G:A):=\{B\in \SubF_r(G:A): 1\in B\}$.
It has  the following important property:\\
\begin{equation}\label{maximal in C(A)}
\SubF^1_r(A)=\Big\{B\subseteq \mathcal{C}^1(A): 1\in B,\;  BB^{-1}\subseteq \mathcal{C}^1(A),
B \mbox{ is maximal with this property}\Big\}
\end{equation}
where $\mathcal{C}(G:A)=\mathcal{C}(A)=\mathcal{C}_\ell(A)$ denotes the complement of $\Dif(A)$ in $G$,
and $\mathcal{C}^1(A)=\mathcal{C}(A)\cup\{1\}$.\\
Because $AB=A\cdot B$ and $1\in B$ if and only if $1\in B\subseteq BB^{-1}\subseteq \mathcal{C}^1(A)$ (by (\ref{SubfactorCondition})),
and the maximality condition in the both are equivalent. \\
If $A\subseteq H\leq G$ and $X$ is a right transversal of $H$ in $G$, then it is proved that (\cite{Hooshmand2020})
\begin{equation}\label{multiplicative subfactor}
\{BX : B\in\mbox{SubF}_r(H:A)\}\subseteq \mbox{SubF}_r(G:A).
\end{equation}
This inclusion cannot be reversed in general. Now, we claim that
\begin{equation}\label{intersection subfactor}
\SubF_r(H:A)=\{\mathcal{B}\cap H: \mathcal{B}\in  \SubF_r(G:A)\}
\end{equation}
and similar property is held for $\SubF^1_r(H:A)$, $\Fac_r(H:A)$, and $\Fac^1_r(H:A)$.\\
For if  $B\in \SubF_r(H:A)$, then $A^{-1}A(BX)=G$ and $A^{-1}A\cap (BX)(BX)^{-1}=\{1\}$,
for all right transversals $X$ of $H$ in $G$ which can assume $1\in X$. Therefore,
$BX\in \SubF_r(G:A)$ (by (\ref{multiplicative subfactor})) and $BX\cap H=B$. Conversely, if  $\mathcal{B}\in \SubF_r(G:A)$, then putting $B=\mathcal{B}\cap H$,
we have
$$BB^{-1}\subseteq (\mathcal{B}\mathcal{B}^{-1})\cap H\subseteq ((G\setminus A^{-1}A)\cup\{1\})\cap H=
(H\setminus A^{-1}A)\cup\{1\}.$$
Also,  $A^{-1}A\mathcal{B}=G$ requires $A^{-1}AB=H$  thus $B\in \SubF_r(H:A)$ (
there is a similar discussion in  \cite{Hooshmand2020} before Theorem 3.12).\\
Again, we claim that
\begin{equation}\label{transversal subfactor}
\SubF_r(H:A)=\{\mathcal{B}x^{-1}\cap H: \mathcal{B}\in  \SubF_r(G:A) , x\in X\}
\end{equation}
by (\ref{intersection subfactor}) and since $\SubF_r(G:Ax)=\SubF_r(G:A)$.
\\
As another important property, if $A\subseteq H\leq G$, $A'\subseteq K\leq G$,
and $\mathcal{C}(H:A)=\mathcal{C}(K:A')$, then $\SubF^1_r(H:A)=\SubF^1_r(K:A')\subseteq 2^{H\cap K}$
(note that $2^X=P(X)$ denotes the power set of $X$).
In particular, if $\Dif(A)=\Dif(A')$, then $\SubF^1_r(G:A)=\SubF^1_r(G:A')$ (by choosing $H=K=G$).\\
Now, for each subset \(A\) of a group \(G\) we assign right subindices of \(A\) as follows:
\[
|G:A|^+ := \sup\{\, |B| : B \in \SubF_r(G:A) \,\} \; : \; \mbox{right upper index of \(A\) (in \(G\))};
\]
\[
|G:A|^- := \inf\{\, |B| : B \in \SubF_r(G:A) \,\} \; : \; \mbox{right lower index of \(A\) (in \(G\))}.
\]

The left notations \(|G:A|_\pm\) are defined analogously.
Note that \[|G:A|^+=\sup\{\, |B| : B\subseteq G,\;  AB=A\cdot B \,\},\] and the above infimum
in the definition of  \(|G:A|^-\) is indeed the minimum.

\begin{definition}[\cite{Hooshmand2020}] \label{BasicDefinition}
 Let \(G\) be an arbitrary group and \(A \subseteq G\). Then
\begin{enumerate}[a)]
\item \( A\) is called right (resp. left) index stable in \(G\) if \(|G:A|^+ = |G:A|^-\) (resp. \(|G:A|_+ = |G:A|_-\)),
	and we use the notation \(|G:A|_r\) (resp. \(|G:A|_\ell\))
	for the common value and call it the right (resp. left) index of \(A\) in \(G\).
\item \( A\) is called index stable (in \(G\)) if all of its four subindices are equal (equivalently \(|G:A|_r = |G:A|_\ell\)),
	and the common value is denoted by \(|G:A|\) and is called the index of \(A\) in \(G\) (a unique cardinal number corresponding to \(A\)).
\item
	Also, \(G\) is called index stable (resp. right index stable) if all its subsets are index stable (resp. right index stable).
\end{enumerate}
	\end{definition}
It is worth noting that if \(G\) is a group and \(H\) a subgroup, then \(H\) (as a subset) is always index stable in \(G\), but
	as an independent group, \(H\) may be not index stable (i.e., it contains a subset that is not index stable in \(H\)).\\
It is interesting and surprising to know that only 14 distinct finite groups are (right, left) index stable
(see \cite{HooshmandYousefian}),
and we show at the end of this section that there is no any infinite (right, left) index stable group.
This fact has already been proved in \cite{Kabenyuk2} (by a different method).
In the next theorem, we restate and improve some basic properties of subindices in arbitrary groups that we need
for our study in infinite groups.
\begin{theorem} \label{BasicTheorem}
 Let $G$ be an arbitrary group and $A\subseteq G$.
Then
\begin{equation}\label{Maininequality}
|G:A|^+|A|\leq |G|\leq |G:A|^-|\Dif(A)|\; , \quad |G:A|^+\leq |\mathcal{C}(A)|+1.
\end{equation}
Moreover:
\begin{enumerate}[a)]

  \item If \(A\subseteq A'\), then \(|G:A|^+\geq|G:A'|^+\) \( (\) but not necessarily \(|G:A|^-\geq|G:A'|^-\)\( )\).\\

  \item \(|G:A|^+\geq |G:\Dif(A)|^+\geq |G:\Dif^2(A)|^+\geq \cdots \geq |G:\Dif^\infty(A)|^+\).\\

  \item If \(\Dif(A)\leq G\) then \(A\) is right index stable and \(|G:A|_r=|G:\Dif(A)|\).\\

  \item If \(A\subseteq H\leq G\), then \(|G:A|^{\pm}\geq|G:H|\) (also \(|G:A|_{\pm}\geq|G:H|\)). Thus if \(|G:A|^+ = |G:H|\),
  then $A$ is right index stable in $G$.

  \item \(|G:\Dif^n(A)|^{\pm}\geq|G:\langle A\rangle|\) and \(|G:\Dif^n(A)|^{\pm}\geq |G:\Dif^\infty(A)|\), for all \(n\geq 0\).\\

  \item If \(|A|>|G\setminus A|\) then \(\Dif(A)=\Dif_r(A)=AA=G\) and so \(|G:A|=1\) (i.e., $A$ is index stable in $G$ with the index 1).\\

  \item If \(|G:\langle A\rangle|=|G|\), then \(A\) is index stable in \(G\) and \(|G:A|=|G|\) (this is trivial for the finite case).\\

\end{enumerate}
\end{theorem}
\begin{proof}
\begin{enumerate}[a)]

  \item See \cite[Theorem 3.24]{Hooshmand2020}.
  \item This is concluded from (a), since
  \[a^{-1}A\subseteq \Dif(a^{-1}A)\subseteq \Dif^2(a^{-1}A)\subseteq \cdots,\]
  \(\Dif^n(a^{-1}A)=\Dif^n(A)\) for all \(a\in A\) and \(n\geq 1\), and \(|G:a^{-1}A|^+=|G:A|^+\).
  \item See \cite[Theorem 3.12(d)]{Hooshmand2020}.
  \item Note that if \(B\) is a right subfactor of \(G\) relative to \(A\), then \(G=\Dif(A)B\subseteq HB\)
  and so \(G=HB\). This implies \(|B|\geq |G:H|\) for all \(B\in\SubF_r(G:A)\).
  \item This is concluded form (d), since \(\Dif^n(A)\subseteq \langle A\rangle\), for all \(n\geq 0\), and
  \[
  |G:\Dif^n(A)|^{\pm}=|G:\Dif^n(a^{-1}A)|^{\pm}\geq |G:\Dif^\infty(a^{-1}A)|^{\pm}=|G:\Dif^\infty(A)|^{\pm}\; ; \; n\geq 0.
  \]
  (It is worth noting that for \(n=0\) we use the identity \(|G:gA|^{\pm}=|G:A|^{\pm}\), for all \(g\in G\).)
  \item See \cite[Theorem 3.12(e)]{Hooshmand2020}.
  \item This is a direct result of (d).
 \end{enumerate}
\end{proof}
\begin{corollary}\label{CountableCor}
If \(G\) is countable and \(|G:\langle A\rangle|\) is not finite, then \(A\) is index stable and \(|G:A|=\aleph_0\).
\end{corollary}

\begin{example}\label{BernouiliSubindexExample} One can determine the subindices of so many subsets
by using the above basic theorem and its corollaries. We mention some interesting examples as follows.\\
\textbf{(a)} If \(H,K\) are subgroups of \(G\) such that \(HK=KH\), then \(H\cup K\) is index stable in \(G\), since
\[(H\cup K)^{-1}(H\cup K) =(H\cup K) (H\cup K)^{-1}=HK\leq G,\] (the intersection is also obviously index stable).\\
For instance, every subset \(A=a\mathbb{Z}\cup b\mathbb{Z}\), where \(a,b\in \mathbb{Z}\), is index stable in \((\mathbb{Z},+)\) and
\(|\mathbb{Z}:a\mathbb{Z}\cup b\mathbb{Z}|=\gcd(a,b)\), because \(a\mathbb{Z}+b\mathbb{Z}=\gcd(a,b)\mathbb{Z}\).\\
\textbf{(b)} The set of irrational numbers is index stable (with index 1) since its difference set is \(\mathbb{R}\).\\
\textbf{(c)} If \(A\subseteq A'\) and \(|G:A|^+=2\), then \(A'\) is right index stable and either \(|G:A'|_r=1\) (equivalently,
\((A')^{-1}A'=G\)) or \(|G:A'|_r=2\). Indeed, putting \(C=A'\setminus A\) we have
\[
\Dif(A')=\Dif(A)\cup \Dif(C)\cup A^{-1}C\cup C^{-1}A,
\]
 and so \(|G:A'|_r=1\) if and only if \(\mathcal{C}(A) \subseteq C^{-1}C\cup A^{-1}C\cup C^{-1}A\).
As an application, since \(|\mathbb{Z}:\{0,1,4,9,\ldots\}|=2\) we have
\(|\mathbb{Z}:\{0,1,4,9,\ldots\}\cup\{ x \}|=2\) for every integer \(x\). Because
\[
\mathcal{C}(\{0,1,4,9,\ldots\})=4\mathbb{Z}+2\nsubseteq \{0,\pm x, \pm (x-1),\pm(x-4),\pm(x-9),\ldots\},
\]
 for every fixed integer \(x\).
But for \(\{0,1,4,9,\ldots\}\setminus\{x\}\) the status is completely different and we have only
\[|\mathbb{Z}:\{0,1,4,9,\ldots\}\setminus\{x\}|^+\geq 2,\]
and it may be non index stable (see Lemma \ref{perfect square}).
\end{example}
\begin{remark}\label{FirstSubindexExample}
Finite subsets of infinite groups \(G\) are index stable with the index \(|G|\) by (\ref{Maininequality})  (see
Theorem~\ref{InfiniteSubindexEquation} for a more general proposition).
All subsets of integers (resp. rationales) are index stable in rationales (resp. reals) with the index \(\aleph_0\) (resp. \(2^{\aleph_0}\)),
by Theorem \ref{BasicTheorem} (d). Therefore, index stability of infinite subsets of \underline{integers in integers} and \underline{rationales in rationales} are challenging.
For instance, the set of perfect squares is index stable with the index \(2\) (in \(\mathbb{Z}\)), but the sequence of positive perfect squares is non-index stable (see Lemma~\ref{perfect square} and \cite{Hooshmand2020}). We will study such interesting topics and problems about numbers in Section 3.
\end{remark}

The inequality $|G:A|^+\leq |\mathcal{C}(A)|+1$ has a related important question that when it takes the equality.
Here we give a characterization for it if $\Dif(A)$ is co-finite.
\begin{theorem}
The following are equivalent:\\
\textbf{$($a$)$} $\mathcal{C}^1(A)\leq G; \;\;\;$  \textbf{$($b$)$} $\mathcal{C}^1(A)\in \SubF^1_r(G:A); \;\;\;$
\textbf{$($c$)$} $\mathcal{C}(A)\mathcal{C}(A)\subseteq C^1(A);$  \\
 \textbf{$($d$)$} $\SubF^1_r(G:A)=\{\mathcal{C}^1(A)\}; \;\;\;$  \textbf{$($e$)$} $\SubF^1_r(G:A)$ is a singleton;
$\;\;\;\textbf{($f$)}$ $A\mathcal{C}^1(A)=A\cdot \mathcal{C}^1(A)$ .\\
Each the above (equivalent) conditions requires $|G:A|^+= |\mathcal{C}(A)|+1$.\\
Also, the condition $|G:A|^+= |\mathcal{C}(A)|+1$  is equivalent to them if $\Dif(A)$ is co-finite.\\
(Note that the condition $\mathcal{C}(A)\mathcal{C}(A)\subseteq C^1(A)$ can be replaced by $\mathcal{C}(A)\mathcal{C}(A)= C^1(A)$ if
$|\mathcal{C}(A)|\geq2$).
\end{theorem}
\begin{proof}
First we may assume $1\in A$, because $\mathcal{C}^1(gA)=\mathcal{C}^1(A)$ and $\SubF^1_r(G:gA)=\SubF^1_r(G:A)$
for all $g\in G$. Thus $A\subseteq \Dif(A)$. Also, we have
$$
\mathcal{C}^1(A)\leq G \iff \mathcal{C}^1(A)\mathcal{C}^1(A)\subseteq C^1(A) \iff \mathcal{C}(A)\mathcal{C}(A)\subseteq C^1(A),
$$
and also equivalent to $\mathcal{C}(A)\mathcal{C}(A)= C^1(A)$ if $|\mathcal{C}(A)|\geq2$.\\
On the other hand the product $A\mathcal{C}^1(A)$ is direct if and only if $\mathcal{C}^1(A)\mathcal{C}^1(A)\subseteq C^1(A)$,
and this implies $$\SubF^1_r(G:A)=\{\mathcal{C}^1(A)\},$$ by (\ref{maximal in C(A)}).
Now, let $\SubF^1_r(G:A)$ be a singleton $\{B\}$. Since the product  $A\{1,\alpha\}$ is direct for every
$\alpha\in \mathcal{C}(A)$, we infer $\{1,\alpha\}\subseteq B$. Therefore, $B=\mathcal{C}^1(A)$.\\
If $\Dif(A)$ is co-finite and $|G:A|^+= |\mathcal{C}(A)|+1$, then $|B|=|\mathcal{C}^1(A)|$
for some $B\in \SubF^1_r(G:A)$,  and so $\mathcal{C}^1(A)=B$ (because $B\subseteq \mathcal{C}^1(A)$
and $\mathcal{C}^1(A)$ is finite).
\end{proof}

We have some information about index stability and subindices of (set-theoretic) subgroup complements.
For example, if $H$ is a non-trivial subgroup of \(G\), then \(|H^c|\geq |H|\) and we have two cases: \(|H^c|> |H|\)
which requires \(|G:H^c|=1\), and   \(|H^c|=|H|\) that is challengeable (especially if $G$ is infinite).
Now we completely characterize its index stability statues dependent on the index of the subgroup.
\begin{theorem}[Characterization of indices of subgroup complements]\label{subgroupcomplement}
For every group \(G\) and subgroup \(H\), the subset $H^c$ is index stable, and we have
\[
|G:H^c|=\left\{\begin{array}{ll} |G|\; &  \; |G:H|=1\\
2\;\;\; & |G:H|=2\\
1\;\;\; & |G:H|>2
\end{array}
\right.
\]
Moreover, if \(|G|>2\), then \(|G:H^c|=|G|\) if and only if \(G=H\), and
\[
|G:H^c|=1 \Leftrightarrow |G:H|>2 \Leftrightarrow H^cH^c=G\; , \;
|G:H^c|=2 \Leftrightarrow |G:H|=2 \Leftrightarrow H^cH^c=H.
\]
\end{theorem}
\begin{proof}
First we have
\[
H^cH^c=\left\{\begin{array}{ll} \emptyset &  \; |G:H|=1\\
H\; & |G:H|=2\\
G\; & |G:H|>2
\end{array}
\right.
\]
Because if \(|G:H|=2\) (resp. \(|G:H|>2\)), then \(H^c=Hg_0\)  for some
 \(g_0\in G\) (resp. for every \(g\in G\) there is \(g'\in G \) such that \(gH\neq g'H\)).
 Thus \(G=Gg_0=Hg_0\cup Hg_0^2=H^c\cup Hg_0^2\) and so \(H^cH^c=Hg_0^2=H\) (resp. \(g=g(g^{-1}g')\in H^cH^c\)).
 Then, since \(\Dif_\ell(H^c)=\Dif_r(H^c)=H^cH^c\), we obtain the results by Theorem~\ref{BasicTheorem}(c) (and
 similar property for the left subindices).

For proving the last equivalent conditions note that
if  \(|G:H^c|=|G|\) then the case \(|G:H|=2\) (resp. \(|G:H|>2\)) gives the contradiction
\(|G|=2\) (resp. \(|G|=1\)). Other remained results can be obtained similarly.
\end{proof}
\begin{example}\label{InfiniteSunindexExample}
The set of odd numbers (resp. irrational numbers) is index stable in integers (resp. reals) with index 2
(resp. index 1). Also, \(\mathbb{Q}\setminus \mathbb{Z}\) is index stable (in \(\mathbb{Q}\)) with index 1, since
\(|\mathbb{Q}:\mathbb{Z}|=\aleph_0>2\).
\end{example}
\begin{remark}\label{Balanced index stable}
The above theorem says every subgroup is ``balanced index stable'' which means both itself and its complement are index stable.
Thus, we can list some (left, right) balanced index stable subsets as follows (it is obvious that \(A\) is
balanced index stable if and only \(A^c\) is so):
\begin{itemize}
\item
All subgroups (and their set-theoretic complements) are balanced index stable.
\item
Every singleton and the empty set (subsets \(A\) such that \(|A|\leq 1\)) are balanced index stable.
\item
Every subset \(A\) of a finite group \(G\) with \(\frac{|G|}{3}<|A|\leq \frac{|G|}{2}\) is
 left and right balanced index stable, and so it is balanced index stable if \(G\) is abelian,
by \cite[Corollary 3.4(b)]{Hooshmand2023} and Theorem~\ref{BasicTheorem}(f).
\item Every finite (and co-finite) subset of an infinite group is balanced index stable (by Theorem~\ref{InfiniteSubindexEquation} and Theorem~\ref{BasicTheorem}(f)).
\end{itemize}
\end{remark}
Motivated by a result of M. Kabenyuk \cite[Lemma~4]{Kabenyuk} the idea is arisen that the inequalities for the right subindices
of subsets of a subgroup (\cite[Theorem~3.22]{Hooshmand2020}) maybe changed to identities (as I had some communications with him about it).
It is interesting to know that it completely agrees with the identity $|G:K| = |G:H|\,|H:K|$ where $K\leq H\leq G$.
\begin{theorem}\label{Subgroupsindex}
If \(A \subseteq H \leq G\), then
\begin{equation}\label{multiplicative subindex identity}
|G:A|^- = |G:H|\,|H:A|^- \; , \;  |G:A|^+=|G:H|\,|H:A|^+ .
\end{equation}
Hence, for every subset \(A\),
\[
|G:A|^- = |G:\langle A\rangle|\,|\langle A\rangle:A|^-
\leq |G:\langle A\rangle|\,|\langle A\rangle:A|^+ = |G:A|^+
\]
(analogously for the left subindices).
\end{theorem}
\begin{proof}
Let $A\subseteq H\leq G$ and $X$ be a right transversal of $H$ in $G$ (equivalently $G=H\cdot X$). If $\mathcal{B}$ is a right subfactor of $G$ relative to $A$
(i.e, $\mathcal{B}\in \SubF_r(G:A)$), then
$$
|\mathcal{B}|=\sum_{x\in X}|\mathcal{B}\cap Hx|=\sum_{x\in X}|\mathcal{B}x^{-1}\cap H|.
$$

Since $\mathcal{B}x^{-1}\cap H \in \SubF_r(H:A)$, for all $x\in X$ (by (\ref{intersection subfactor}),
and the identity $\SubF_r(G:A)g=\SubF_r(G:A)$) and
we conclude that $|H:A|^-\leq |\mathcal{B}x^{-1}\cap H|\leq |H:A|^+$ and so the above equality requires
$$
|X||H:A|^-\leq |\mathcal{B}|\leq |X||H:A|^+ .
$$
Therefore
$$
|G:H||H:A|^-\leq |G:A|^-\leq |G:A|^+\leq |G:H||H:A|^+ ,
$$
since the above inequalities hold for all $\mathcal{B}\in \SubF_r(G:A)$, and $|X|=|G:H|$.
Thus we arrive at (\ref{multiplicative subindex identity}) due to the inequalities
$|G:A|^-\leq |G:H||H:A|^-\leq|G:H||H:A|^+ \leq |G:A|^+$ from \cite[Theorem~3.22]{Hooshmand2020}.
\end{proof}

\begin{corollary}\label{SubindexsubgroupEquation}
Let \(A \subseteq H \leq G\).
\begin{enumerate}[(a)]
\item if $A$ is right index stable in \(G\), then
\[
|G:A|_r = |G:H|\,|H:A|^- = |G:H|\,|H:A|^+.
\]
\item if $A$ is right index stable in \(H\), then $A$ is right index stable in \(G\) and we have
\[
|G:A|_r = |G:H|\,|H:A|_r.
\]
But in general the converse is not true e.g., $|\mathbb{Q}:\mathbb{Z}|=\aleph_0$, $3=|\mathbb{Z}:\{1,4,9,\ldots\}|^-\neq |\mathbb{Z}:\{1,4,9,\ldots\}|^+=4$, and $|\mathbb{Q}:\{1,4,9,\ldots\}|=\aleph_0$(see Lemma~\ref{perfect square}).
\item
As a weak converse of (b), the right index stability of $A$ in $G$ implies right index stability of $A$ in $H$ if one of the following
conditions hold:\\
- $|G:H|$ is finite (i.e, $|G:H|<\aleph_0$);\\
- $|G:H|$ is infinite and $|G:H|\leq |H:A|^-$ (i.e, $\aleph_0\leq |G:H|\leq |H:A|^-$).
\item if $|G:H|\,|H:A|^- \neq |G:H|\,|H:A|^+$, then $A$ is non index stable in $G$.
\end{enumerate}
\end{corollary}
\begin{remark}\label{Retables}
Theorem~\ref{Subgroupsindex} and Corollary~\ref{SubindexsubgroupEquation} give us some strong results for the both theoretical and computational. For instance:
\begin{itemize}
\item if $A$ is right index stable in $\langle A\rangle$, then it is right index stable in $G$. This is very useful for
the computational aspect if $A$ is not a generating set.
\item for checking the index and $k$-index stability of a group it is enough to check subsets $A$ such that don't lie
in an index stable subgroup.
\end{itemize}
Also, they provide various cases for determining the subindices and related index stability .
They have the main role for characterization of finite index stabile groups, since
every group with a (right) non-index stable subgroup hazing finite index is not (right) index stable  \cite{Hooshmand2020}.
 Of course one can also use these important results for infinite groups.\\
We have designed two tables in the last section so that we can compare different cases and extract the necessary results.
For example, Table~\ref{table3} shows that if $|G:H|$ is infinite and $A$ is right index stable in $G$, then the equivalent conditions $|H:A|^->|G:H|$ and $|H:A|^+>|G:H|$
imply $A$ is also right index stable in $H$.
\end{remark}

\begin{example}\label{TableExample}
The group $G=\mathbb{Z}_2\times \mathbb{Z}_2\times\mathbb{Z}_2\times\mathbb{Z}_2$ is index stable (see \cite{HooshmandYousefian}). Hence, for every
subgroup $H$ and subsets $A$ of $H$ we have $|G:A|=|G:H||H:A|$, which means in this group the well-known identity $|G:K|=|G:H||H:K|$
($K\leq H$) is also valid for all $A\subseteq H$. \\
Now, let $G$ be an arbitrary group and $H\leq G$. If $A$ is a subset of $H$ with $|G:A|^+=|G:H|$, then
$A$ is right index stable in $G$ with the right index $|G:H|$. Moreover, if $|G:H|$ is finite, then ($A$ is also right index stable in $H$ and) $|H:A|_r=1$ and so
$A^{-1}A=H$ which means $A$ is a difference-generating set of $H$. But if $|G:H|$ is infinite, then $A$ may be non right index
stable in $H$ (only we have $|H:A|^{\pm} \leq |G:A|_r=|G:H|$). Because putting $G=\mathbb{Q}$, $H=\mathbb{Z}$, and \(A=\{1,4,9,\ldots\}\) we have
$|G:A|^+=|G:H|=\aleph_0$ and $A$ is not index stable in $H$.
\end{example}
\subsection{Relations of subindices to packing and covering numbers}
The \emph{(left) packing number} of $A$ in $G$ is defined as:

    \[
   pack_G(A) = \sup\bigl\{ |S| : S \subseteq G, \{xA\}_{x \in S} \text{ is disjoint} \bigr\}.
    \]

The notion of packing numbers is well-developed in the context of Polish groups, where descriptive set theory provides cardinal restrictions on the possible values (see \cite{Protasov2011}).\\

On the other hand, the concept of a covering number quantifies the minimum efficiency required to cover an entire group using translated copies of a specific subset.
The (left) covering number of $A$ in $G$, often denoted $\mathrm{cov}_G(A)$ or $\kappa_G(A)$, is defined by
$$
\mathrm{cov}_G(A) = \min\bigl\{ |X| : X \subseteq G, \, G = XA \bigr\}
$$
The covering number is a fundamental measure of the additive structure of the subset $A$ relative to $G$. \\
Now, we introduce an interesting and very importanr relation among the left upper and lower subindices, (left) packing number, and (left) covering number  of a subset.
\begin{lemma}\label{covering and packing}
For every subset $A$ of a group $G$ we have
\begin{equation}\label{covering and packing identity}
\mathrm{cov}_G(AA^{-1})\leq |G : A|_-\leq |G : A|_+=\mathrm{pack}_G(A)
\end{equation}
\end{lemma}
\begin{proof}
First, we have
\[
|G : A|_+
= \sup\{|B| : B \in \SubF_\ell(A)\}
= \sup\{|B| : B \subseteq G, \, BA = B \cdot A, \, BAA^{-1}=G\}
\]
\[
= \sup\{|B| : B \subseteq G, \, BA = B \cdot A\}
= \sup\{|B| : B \subseteq G, \, \{bA \}_{b \in B} \text{ is disjoint}\}
= \mathrm{pack}_G(A).
\]
The third equality is obtained from the fact that if $BA = B \cdot A$, then $B$ lies in a left
subfactor of $G$ relative to $A$ (see \cite[Corollary~3.3]{Hooshmand2020}).
Also,
$$
\mathrm{cov}_G(AA^{-1})=\min\bigl\{ |X| : X \subseteq G, \, G = XAA^{-1} \bigr\}\leq
\min\bigl\{ |X| : X \subseteq G, \, G = XAA^{-1} , \, XA=X\cdot A \bigr\}=|G : A|_-.
$$
Thus we obtain (\ref{covering and packing}), since always  $|G : A|_-\leq |G : A|_+$.
\end{proof}

This connection helps the two topics of subindices and packing numbers to improve each other (or maybe for the covering numbers).
As some instances, we apply some results of many papers on small subsets of groups and related packing numbers to
answer some questions of our topic, in the next part.
Also, in contrast, the results obtained in the pervious articles and the continuation of this paper can contribute to the topic
packing numbers and related topics. For instance, the RSFA algorithm for infinite groups (in Subsection 2.2) may be useful
for the study of small and large subsets, and related packing numbers.
For example, we obtain that the packing number
of perfect square numbers (resp. positive perfect square numbers) in $\mathbb{Z}$ is 2 (resp. 4, see Example~\ref{square}).
\subsection{Answering to some open questions and unsolved problems} There are several basic and important problems, questions,
conjectures, and research project about the topic of subindices, subfactors, and index stability of groups \cite{Hooshmand2020, Hooshmand2023}.
Many of them have been answered at the end of \cite{HooshmandYousefian} (In fact, some of them had simple answers that had not been given enough attention before).
Now, we consider some of them here.\\
\cite[\textbf{Problems and questions III}]{Hooshmand2020}. \\
{\bf (a)} Characterize all $n$ such that $S_n$ ((respectively, $A_n$, $D_{2n}$, etc.) is index stable. \\
\textbf{Solution.} $S_n$ and $A_n$; $n=1,2,3$, $D_{2n}$; $n=1,2,3,4$, $\mathbb{Z}_{n}$; $n=1,2,3,4,5,7$.\\
Indeed, it is proved that the class of all (right/left) finite index stable groups consists exactly of all groups of order $\leq 9$ except $C_6$, $C_8$, $C_9$, together with $C_2\times C_2\times C_2\times C_2$ (see \cite[Theorem~3.1]{HooshmandYousefian} ).\\
{\bf (b)} Characterize or classify all index stable subgroups of $(\mathbb{R},+)$ and $(\mathbb{C},+)$.
Especially, is rational (resp. real, complex) numbers group index stable?\\
\textbf{Solution.} Only the identity subgroup (in the sequel, we show that no infinite group is index stable).\\
{\bf (c)} Give some subsets of some (finite, infinite) groups such that all its sub-indexes are different (i.e.,
it has no any type of index stability). \\
It is still \textbf{open}. Also see \cite[Section~4]{HooshmandYousefian}.\\
{\bf (d)} Give finite and infinite examples of a group that is right (resp. left) but not left (resp. right) index stable.\\
\textbf{Nothing.} See \cite{Hooshmand2023} (after Definition 3.9).\\
Also, give finite and infinite examples of groups that are both left and right index
stable but not index stable. \\
\textbf{Nothing} for finite groups. But, for infinite groups it is still \textbf{open}.\\
{\bf (e)} Let $\alpha\geq \aleph_0$ be a given cardinal number. Prove or disprove: \\
\textbf{(e1)} There exist an index stable
(resp. non-index stable) abelian and also non-abelian group of order $\alpha$. \\
\textbf{Solution.} There is no any infinite index stable group. In \cite{BanakhLyaskovska2006, BanakhLyaskovska2008}, T. Banakh and N. Lyaskovska
proved that each infinite (abelian and non-abelian) group \(G\) contains a subset \(A\) such that,
for every positive integer \(n\), there are \(n\) pairwise disjoint (left, right) translation copies of \(A\) in \(G\) but there is no infinitely many such disjoint copies.
Therefore, $\{|B|: B\subseteq G, BA=B\cdot A\}=\mathbb{Z}_+$ (analogously for the right case). But
\[
|G : A|_+= \mathrm{pack}_G(A)
= \sup\{|B| : B \subseteq G, \, BA = B \cdot A\}
= \sup \mathbb{Z}_+=\aleph_0.
\]
 On the other hand, putting $S=\{|B| : B \in \SubF_\ell(A)\}$ we have
 $$
 \emptyset\neq S\subseteq \{|B|: B\subseteq G, BA=B\cdot A\}=\mathbb{Z}_+.
  $$
 Thus, $|G : A|_-=\min S<\aleph_0$. This means $|G : A|_+\neq |G : A|_-$ and so $A$ is not left index stable. Hence
 $G$ is neither left nor right index stable, due to the part (d). In particular, it is not index stable.
\begin{remark}

The above result also give negative answer to the question \cite[21.58]{KourovkaNotebook} which asks:
is there an infinite group $G$ such that every subset $A\subseteq G$ satisfies the
property that all the maximal subsets $B$ for which the product $AB$ is direct
have the same cardinality?\\
Notably, \cite{Kabenyuk2} provided an alternative method for answering this question.
Also, we conclude that there are only 14 (right/left) index stable groups among all groups that all of them are finite (as mentioned in the above).
\end{remark}
 \textbf{(e2)} There exist an abelian and also a non-abelian group $G$ with a subset
$A$ such that $|G:A|^-=2$ and $|G:A|^+=\alpha$. Moreover, there are two finite and infinite such
groups if $\alpha$ is finite.\\
It is still \textbf{open}.\\
{\bf (f)} Is there any group $G$ with a subset $A$ satisfying one (both) of
the following properties?\\
(1) $|G:A|^+\notin\{|B|: B\in \SubF_r(A)\} \qquad$  (2) $|G:A|_+\notin\{|B|: B\in \SubF_\ell(A)\}$.\\
\textbf{Answer.} \textbf{Yes} (see (e1)).\\
{\bf (g)} Let $G$ be a group and $A\subseteq G$ with $Dif_\ell^\infty(A)=G$. Prove or disprove:\\
(g1) If $dl^\infty_\ell(A)$ is finite then $|G:A|^+$ or $|G:A|^-$  is also finite; \\
It is still \textbf{open}.\\
(g2) If $dl^\infty_\ell(A)=\infty$ and $G$ is infinitely countable then $|G:A|_r=\aleph_0$,
and visa versa;  \\
It is still \textbf{open}.\\
(g3)  The statements (g1) and (g2) are true if $\langle A \rangle_\ell^{(N)}=G$
, where $N=dl^\infty_\ell(A)$ (for the case that either (g1) or (g2) is false).\\
It is still \textbf{open}.\\
{\bf (h)} If $Dif_\ell^2(A)=G$ and $A$ is right index stable then $|G:A|_r=2$,
and visa versa (analogously for the left and two-sided cases).\\
\textbf{Disproved.} See \cite[Section~4]{HooshmandYousefian}. \\
{\bf (i)} If $A\subseteq H\leq G$ and $A$ is index stable in $G$, then it is so
in $H$ (and visa versa).\\
\textbf{Solution.} This is \textbf{false in general} (it is true if $|G:H|<\aleph_0$,
or $\aleph_0\leq |G:H|\leq |H:A|^-$), but the \textbf{converse is true} by Corollary \ref{SubindexsubgroupEquation}.\\
The above results also enable us to solve some related problems. For example:\\
\cite[\textbf{Problems and questions III}]{Hooshmand2020}. Are all products of $\mathbb{Z}_{2}$ index stable?
What about $\mathbb{Z}_{3}$ , $\mathbb{Z}_{4}$ , $\mathbb{Z}_{5}$ and $\mathbb{Z}_{7}$ ? \\
\textbf{Answer.} No. Indeed, the powers of $\mathbb{Z}_{n}$ ($n>1$) that are not index stable are as follows: $\mathbb{Z}_{2}^m$; $m\geq5$,
$\mathbb{Z}_{3}^m$; $m\geq3$, $\mathbb{Z}_{4}^m$, $\mathbb{Z}_{5}^m$, and $\mathbb{Z}_{7}^m$; $m\geq2$,
and others $\mathbb{Z}_{n}^m$; $m\geq1$.
\section{More sub-index properties for infinite groups}
In \cite{Hooshmand2023, HooshmandYousefian} we studied subindex and subfactor properties in finite groups. In the previous section we consider
some basic and important properties in arbitrary group. For infinite groups, we observed that every finite subset \(A\) is index stable
and \(|G:A|=|G|\). For infinite subsets \(A\)
we have \(|Dif_\ell(A)|=|Dif_r(A)|=|A|\) (because \(|A|\leq|Dif(A)|\leq |A^{-1}||A|=|A|^2=|A|\)) and the following important fact shows that the only challenging case for index stability of subsets is \(|A^c|=|A|\) (although for finite groups is \(|A^c|\geq |A|\), where  \(A^c=G\setminus A\)).
Hence, we call a subset \(A\) of a given group \(G\) \emph{``cardinally balanced''} if \(|A^c|=|A|\). If \(G\) is an infinite group, then \(A\) is cardinally balance if and only if \(|A^c|=|A|=|G|\).
\begin{theorem}\label{InfiniteSubindexEquation} Let \(G\) be an infinite group.
\begin{enumerate}[a)]
\item For every nonempty subset \(A\) we have
\begin{equation}\label{InfiniteSubindexidentity}
|G|=|G:A|^-|A|=|G:A|^+|A|=|G:A|_-|A|=|G:A|_+|A|
\end{equation}
\item All subsets \(A\) with \(|A^c|\neq |A|\) of \(G\) are index stable.

Moreover,
we have \(|G:A|=|G|\) if \(|A|<|G|\), and \(|G:A|=1\) if \(|A|=|G|>|A^c|\).
\end{enumerate}
Therefore, the (left, right, and two-sided)
non index stable subsets of infinite groups (if exist) are among the cardinally balanced subsets.
\end{theorem}
\begin{proof}
Let \(G\) be an infinite group, \(\emptyset\neq A\subseteq G\), and \(Dif(A)B=G\). Then,
\[
|G|=|Dif(A)B|=|Dif(A)||B|=|A||B|.
\]
For if  \(A\) is infinite then \(|\Dif(A)|=|A|\)
and if \(A\) is finite then \(\Dif(A)\) is also finite.
In both cases we have \(|Dif(A)||B|=|A||B|\).
Therefore \(|G|=|A||B|\), for all \(B\in \SubF_r(A)\), and so \(|G|=|G:A|^-|A|=|G:A|^+|A|\). Similarly, we have
\(|G|=|G:A|_-|A|=|G:A|_+|A|\).\\
Now, the identity (\ref{InfiniteSubindexidentity}) together with Theorem~\ref{BasicTheorem}(f)  imply (b).
\end{proof}
The set of irrational (\(\mathbb{Q}^c\)), transcendental, and algebraic numbers are index stable (in \(\mathbb{R}\))
since they are not cardinally balanced.
\begin{corollary}\label{finiteSunindexCriteria1}
Let \(G\) be an infinite group and \(A\) a given subset of \(G\) . Then,
if \(A\) (resp. \(\Dif(A)\)) is co-finite then it is index stable and \(|G:A|=1\).
(resp. \(|G:A|^\pm\) are finite).
\end{corollary}
Regarding part (b) the above corollary, one can always find a subset of every group that its difference is co-finite.
\begin{proposition}\label{codif}
Let $G$ be an infinite group and $F$ be a subset of $G$ such that $1\not\in F=F^{-1}$ and $|F|<|G|$.
Then there exists a subset $A$ of $G$ such that $\mathcal{C}(A)=F$ (i.e.,  $\Dif(A)=G\setminus F$).
\end{proposition}

\begin{proof}
Let $\kappa=|G|$ and $\{g_\alpha:\alpha<\kappa\}$ be an injective enumeration of the set $G\setminus F$.
We shall proceed by transfinite induction on $\alpha$.
Let $\alpha<\kappa$ be an ordinal.
Put $A_\alpha=\{a_\beta,a_\beta g_\beta^{-1}:\beta<\alpha\}.$
The inductive construction ensures that $A_\alpha\cap A_\alpha F=\varnothing$.
Since $|F|,|A_\alpha|<|G|$, we can pick an element
$a_\alpha\in G\setminus\left(A_\alpha F\cup A_\alpha F g_\alpha\right)$.
Then $$(A_\alpha\cup \{a_\alpha,a_\alpha g_\alpha^{-1}\})^{-1}
(A_\alpha\cup \{a_\alpha,a_\alpha g_\alpha^{-1}\})\cap F=\varnothing,$$ and
$g_\alpha=(a_\alpha g_\alpha^{-1})^{-1}a_\alpha$.
Finally, put $A=\{a_\alpha,a_\alpha g_\alpha^{-1}:\alpha<\kappa\}$.
\end{proof}

\subsection{Subindices of subsets with non \underline{syndetic} differences.}

In group theory and additive combinatorics, syndetic sets are important examples of ``large'' sets. A subset $A$ of an infinite group $G$ is called \emph{right syndetic} if there is a finite set $F \subseteq G$ such that $AF = G$. This means that a finite number of right-translates of $A$ cover the whole group (so $A$ must be infinite).

Syndetic sets relate to other concepts like thickness and piecewise syndeticity (see \cite{Syndetic}). These ideas are essential for proving key results, such as van der Waerden's and Szemer\'edi's theorems, using the Furstenberg correspondence principle~\cite{Furstenberg book, Furstenberg paper, GlasscockLe2024}. Syndetic sets also connect different fields: they appear as return-time sets in ergodic theory and as members of minimal ideals in the Stone--\v{C}ech compactification $\beta G$~\cite{HindmanStrauss1998}.

For the integers, a set is syndetic if and only if it has bounded gaps. This means there is a number $N \in \mathbb{N}$ such that every interval of length $N$ contains an element of the set (we will discuss this in Section 3).

Finally, adding or removing a finite number of elements does not change whether a set is syndetic
(see \cite{TaoVu}). For any subset $S \subseteq G$ and finite set $F \subseteq G$, we have:
\[
    S \text{ is right syndetic }
    \iff S \cup F \text{ is right syndetic }
    \iff S \setminus F \text{ is right syndetic}.
\]



The next corollary is an immediate result of Theorem \ref{InfiniteSubindexEquation}, \ref{Maininequality}, and \ref{SubindexsubgroupEquation}
(also see \cite[Theorem 3.8]{Hooshmand2020}).
\begin{corollary}\label{InfiniteSunindexCriteria2}
Let \(G\) be an infinite group and \(A\) a given subset of \(G\) . Then
 if one of the following conditions holds
then \(|G:A|^\pm\) are infinite \( \)(i.e., \(|G:A|^-\geq \aleph_0)\):
\begin{enumerate}[a)]
\item \(\Dif(A)\) is not right syndetic;
\item  the equation \(\Dif_\ell(A)\cap \Dif_r(X)=\{1\}\) has no any finite maximal solution in \(2^G\);
\item the system of equations
\[
\left\{\begin{array}{l} Dif(A)X=G    \\
\Dif_\ell(A)\cap \Dif_r(X)=\{1\}
\end{array}
\right.
\]
has no any finite solution in \(2^G\).
\end{enumerate}
\end{corollary}

\begin{example}\label{InfiniteSunindexExample}
Every bounded subset of rational numbers is index stable (in \(\mathbb{Q}\)) with index \(\aleph_0\)
by Corollary \ref{InfiniteSunindexCriteria2}.
\end{example}
Given the significance of the syndeticity of $\operatorname{Dif}(A)$, the following lemma establishes an analogue
 of the preceding equivalences specifically for difference sets. This result is instrumental for determining the
 subindices of subsets in infinite groups, particularly in the countable case.
\begin{lemma}\label{DifferenceSyndeticProperty}
Let \(G\) be an infinite group, \(g\in G\), \(A\subseteq G\), and \(F\) a finite subset of \(G\). Then
\begin{enumerate}[a)]
\item if  \(g^{-1}Ag\subseteq A\), then
\begin{equation}\label{ab}
\Dif(A\cup \{g\}) \text{ is right syndetic }
\;\Leftrightarrow\; \Dif(A) \text{ is right syndetic }
\;\Leftrightarrow\; \Dif(A\setminus \{g\}) \text{ is right syndetic}.
\end{equation}
\item if  \(f^{-1}Af\subseteq A\) for every \(f\in F\), and if all elements of \(F\)
commute with each other, then
\begin{equation}\label{abc}
\Dif(A\cup F) \text{ is right syndetic }
\;\Leftrightarrow\; \Dif(A) \text{ is right syndetic }
\;\Leftrightarrow\; \Dif(A\setminus F) \text{ is right syndetic}.
\end{equation}
\end{enumerate}
\end{lemma}
\begin{proof}
It is clear that if  $\Dif(A)$ is right syndetic, then $\Dif(A\cup F)$ is so (since $\Dif(A)\subseteq\Dif(A\cup F)$).
Now, let $\Dif(A\cup \{g\})Y=G$ and $1\in A\cap Y$ (without loss of the generality) for some finite subset $Y$.
Putting $Y_g=(Y\cup \{g,g^{-1}\})Y$ we have
$$
\Dif(A)Y_g=\Dif(A)YY \cup \Dif(A)gY\cup\Dif(A)g^{-1}Y\supseteq
\Dif(A)Y \cup A^{-1}gY\cup Ag^{-1}Y$$
$$\supseteq (\Dif(A) \cup A^{-1}g\cup g^{-1}A)Y
=\Dif(A\cup \{g\})Y=G
$$
Thus $\Dif(A)$ is syndetic since $Y_g$ is finite.\\
If $\Dif(A)$ is syndetic and $g\in A$ (for the case $g\notin A$ it is obvious), then $A=(A\setminus\{g\})\cup\{g\}$, and \(g^{-1}(A\setminus\{g\})g\subseteq A\setminus\{g\}\)
(by the assumption \(g^{-1}Ag\subseteq A\)) so the above argument requirers that $A\setminus\{g\}$ is also syndetic.\\
We prove (b) by using the induction on $|F|$. If $|F|=1$, then it is true by part (a). Suppose that (\ref{abc}) is true for
all finite subsets $F$ of size $n$. If $\mathcal{F}$ is a set of size $n+1$ with the mentioned conditions
and $f_0\in \mathcal{F}$, then \ref{ab} holds for $F=\mathcal{F}\setminus \{f_0\}$. Now,
if $\Dif(A\cup\mathcal{F})$ is right syndetic, then $\Dif((A\cup F)\cup\{f_0\})$ is right syndetic and
$$f_0^{-1}(A\cup F)f_0=f_0^{-1}Af_0\cup F\subseteq A\cup F,$$ and so part (a) implies $\Dif(A\cup F)$ is right syndetic.
This requires $A$ is also right syndetic (by the hypothesis of the induction).\\
This time, let the conditions of part (b) hold and $\Dif(A)$ is right syndetic, then $\Dif((A\setminus F)\cup (F\cap A))$
is syndetic and $f^{-1}(A\setminus F)f\subseteq A\setminus F$ for all $f\in F\supseteq F\cap A$ (for if $f^{-1}\alpha f=f'$
for some $f,f'\in F$ and $\alpha\in A\setminus F$, then $\alpha f=ff'=f'f$ and so $\alpha=f'\in F$ that
is a contradiction), and hence $\Dif(A\setminus F)$ is right syndetic by the above argument. Now one can complete the proof easily.
\end{proof}
\begin{corollary}\label{abeliansyndetic}
If $G$ is abelian and $\Dif(A)$ is syndetic, then $\Dif(A\cup F)$ and $\Dif(A\setminus F)$ are also
 syndetic for all finite subsets $F$.
\end{corollary}
\begin{remark}
The above results give us a main criteria for finiteness status of subindices of subsets of infinite groups by looking at
the syndecity of $\Dif(A)$   as follows.
Let $A$ be a subset of an \underline{abelian} infinite group $G$. Then
\begin{enumerate}[a)]
\item if $\Dif(A)$ is non right syndetic, then $A$ and all its subsets have infinite right subindices as well as
 upper sets of $A$ which have only finitely many elements more than $A$.
Moreover, the mentioned subsets of this case are right index stable with the right index $\aleph_0$
if $G$ is countable (this is very useful for the studies in infinitely countable groups, especially integers, as we see later).
\item if $\Dif(A)$ is right syndetic, then all cases of finiteness and infiniteness of right subindices,
and also right index stability or non-right index stability may be occurred that is the more challengeable case
and we should apply other tools (see the next section).
\end{enumerate}
\end{remark}

\subsection{The RSFA algorithm for infinite groups}
In \cite{Hooshmand2023} (the last section) we introduced a complete algorithm for finding all right sub-factors in finite groups. Now we want to find a way for applying it for infinite groups.
For the order, we look at the basic idea of the algorithm. It says if \(G\) is an arbitrary group, \(A,B\subseteq G\), \(AB=A\cdot B\),  and \(x\notin B\), then
\begin{equation}
A(B\cup \{x\})=A\cdot (B\cup \{x\})\Leftrightarrow x\notin A^{-1}AB\Leftrightarrow x\in \bigcap_{b\in B}\mathcal{C}(A)b
\Leftrightarrow  B\subseteq \mathcal{C}(A)x \Leftrightarrow A^{-1}Ax\subseteq G\setminus B.
\end{equation}
Since \(A\{g_0\}=A\cdot \{g_0\}\) for every fixed element \(g_0\in G\) and \(A\subseteq G\), we can extend this direct product
to \(A\{g_0,g_1,\cdots\}=A\cdot \{g_0,g_1,\cdots\}\) as much as possible. Hence, according the method of Theorem 4.1 (an algorithm for finite groups) of \cite{Hooshmand2023}, we can present the algorithm by two versions, namely RSFA:\\

\underline{The union version}.
Let \(g_0 \in G\) be chosen arbitrary. Let \(g_0,...,g_n,\ n\geq0\) be already constructed. If \(A^{-1}Ag_0\cup...\cup A^{-1}Ag_n=G\) then the algorithm is ended otherwise, continue by choosing an element \(g_{n+1}\notin A^{-1}Ag_0\cup...\cup A^{-1}Ag_n\). Then, we obtain a (finite or infinite) sequence \(\gamma=\{g_n\}_{n\geq 0}\) with the next properties (see Theorem~\ref{RSFAProperties}).
It is interesting to know that here the union \(Ag_0\cup...\cup Ag_n\)
is certainly a disjoint union (but \(A^{-1}Ag_0\cup...\cup A^{-1}Ag_n\) is not necessarily disjoint). \\
\underline{The intersection version}. Let \(g_0 \in G\) be chosen arbitrary. Let \(g_0,...,g_n,\ n\geq0\) be already constructed. If \(\mathcal{C}(A)g_0\cap...\cap \mathcal{C}(A)g_n=\emptyset\) then the algorithm is ended otherwise, continue by choosing an element \(g_{n+1}\in \mathcal{C}(A)g_0\cap...\cap \mathcal{C}(A)g_n\). Then, we obtain the sequence \(\gamma=\{g_n\}_{n\geq 0}\).\\
In \cite{Hooshmand2023}, the intersection version, \(\mathcal{C}(A)g_0\cap...\cap \mathcal{C}(A)g_n\) is denoted by
\(\mathcal{C}^{(n)}_\ell (A)\). It is obvious that \(\mathcal{C}^{(n)}_\ell (A)\subseteq \mathcal{C}(A)\cap \cdots\cap\mathcal{C}(A)^n\) if
\(g_0=1\).\\
Regarding the RSFA, denote by \(N\) the least integer \(N\geq 0\) such that \(A^{-1}Ag_0\cup...\cup A^{-1}Ag_N=G\) (equivalently
\(\mathcal{C}(A)g_0\cap...\cap \mathcal{C}(A)g_N=\emptyset\))  if exists
(such an \(N\) certainly exists in finite groups), and otherwise set \(N=\aleph_0\).
We call such numbers \(N\) (resp. sequences \(\gamma\)) RSFA-numbers (resp. RSFA-sequences) of $A$ (in $G$).
Therefore \(\gamma=\{g_n\}_{n\geq 0}=\{g_n\}_{n=0}^{N}\) and \(|\gamma|=N+1\)
(\(|\gamma|\) is the length of \(\gamma\), and indeed the cardinal number of its image).
Hence an integer number $m$ is a RSFA number if and only if there is a RSFA sequence of the length $m+1$. \\
A RSFA sequence \(\gamma=\{g_n\}_{n\geq 0}\) is called maximal if \(\bigcap_{n\geq0}\mathcal{C}(A)g_{n}=\emptyset\) or equivalently
\[
A^{-1}A\{g_0,\ldots,g_n,\ldots\}=G.
\]

\begin{lemma}\label{ComlpeteRSFA}
Let \(G\) be a group and \(A\subseteq G\). A countable subset \(B\) is a right subfactor of \(G\) relative to \(A\) if and only if
it is the image of a maximal RSFA sequence \(\gamma\) (of \(A\)). Hence,
\begin{itemize}
\item
all the countable right sub-factors can be obtained by the RSFA,
\item
the RSFA is a complete algorithm (i.e., all right sub-factors can be obtained by it) if \(G\) is countable.
\end{itemize}
\end{lemma}
\begin{proof}
Let $B$ be a countable element of $\SubF_r(A)$ and represent it by $B=\{b_n\}_{n=0}^{m}$
where $m$ is either a non-negative integer or $\aleph_0$.
 Now in the algorithm choose $g_0=b_0$ (since $g_0$ is arbitrary in it).
		If $A^{-1}A=G$, then  $|B|=1$, $N=0$ (the related RSFA number) and
		$B=\{b_0\}=\{g_0\}$ thus we are done. Otherwise, suppose that $g_0,\ldots,g_n$ take
		the values $b_0,\ldots,b_n$, respectively,  for
		some $n<m+1$. Since the product $A (\{b_0,\ldots , b_{n}\}\cup \{b_{n+1}\})$ is direct,
		$$
		b_{n+1}\in \bigcap_{i=0}^n\mathcal{C}(A)b_i=\bigcap_{i=0}^n\mathcal{C}(A)g_i=(\bigcup_{i=0}^nA^{-1}Ag_i)^c.
		$$
		So $g_{n+1}$ can take the value $b_{n+1}$ in the algorithm process.
		Therefore $B$ is a subset of (the image of) $\gamma=\{g_n\}_{n\geq 0}$ and so $B=\{g_0,\ldots,g_n,\ldots\}$
(the image of $\gamma$) since
$B$ is maximal. This also implies that $\gamma$ is a maximal RSFA sequence.

Conversely, (the image of) every maximal RSFA sequence \(\gamma\) is a right subfactor of \(G\) relative to \(A\),
by  \cite[Lemma 3.6]{Hooshmand2020} and the above lemma.
\end{proof}
\begin{corollary}\label{ImageSubset}
Let \(B\) be a nonempty countable subset of \(G\). Then, \(AB=A\cdot B\) if and only if \(B\) is a subset of
the image of a RSFA sequence of \(A\) (this existed RSFA sequence can be maximal if \(G\) is countable) .
\end{corollary}
\begin{proof}
This is a direct result of  \cite[Corollary 3.3]{Hooshmand2020} and the above lemma.
\end{proof}
\begin{example}
Let $G=\mathbb{Z}$ and $A=\{0,1\}$. Then \(\gamma=\{g_n=2n\}_{n\geq 0}\) is a RSFA sequence related to $A$ that is not maximal,
since $A-A=\{-1,0,1\}$ and  \(g_{n+1}=2n+2\notin\{-1,0,1\}+g_0\cup \ldots\cup \{-1,0,1\}+g_n\)
and \(\{-1,0,1\}\cup \ldots\cup \{-1,0,1\}+2n\cup\ldots\neq \mathbb{Z}\). But one can check that the sequence
$$
g'_n=(-1)^{n+1}2\lfloor \frac{n+1}{2}\rfloor
$$
is a maximal RSFA sequence related to $\{0,1\}$.\\
Also, the set of all odd numbers is an example of a subset of the infinite group $G=\mathbb{Z}$ such that all the
RSFA numbers are finite (and so maximal). \\
Note that if $G=\mathbb{R}$, then non of the RSFA sequences related to the subset $A=\{0,1\}$ are maximal.\\
\end{example}
The above lemma and corollary require the following properties easily.
\begin{theorem}\label{RSFAtheorem}
Let \(G\) be an arbitrary group, \(A\subseteq G\), and \(\{g_n\}_{n\geq 0}\),
\(N\) are as mentioned in the RSFA. Then
\begin{enumerate}[a)]
\item \(A\{g_0,\cdots,g_n\}=A\cdot \{g_0,\cdots,g_n\}\) for all \(0\leq n\leq N\), and \(\{g_0,\cdots,g_n\}\) is not a right sub-factor if \(n<N\).\\
\item \(\{g_0,\cdots,g_N\}\) is a right sub-factor related to \(A\) if and only if \(\gamma=\{g_n\}_{n=0}^{N}\) is maximal.
Hence, \(\{g_0,\cdots,g_N\}\) is a right sub-factor related to \(A\) if \(N\) is finite (e.g., if $|G:A|^+$ is finite). \\
\item \(|G:A|^-\leq \aleph_0\) if and only if there is a maximal RSFA-sequence. Hence \(|G:A|^-> \aleph_0\) if and only if
all the RSFA-sequences are \((\)infinite and\()\) not maximal.\\
\item  Always we have
\[
|G:A|^+\geq|G:A|^-\geq \min\{N+1: \mbox{\(N\) is a RSFA-number} \}
\]
Thus, if there is a maximal RSFA-sequence, then
\[
|G:A|^-=\min\{|\gamma|: \; \gamma \; \mbox{is a maximal RSFA sequence} \}=\min\{N+1: \mbox{ \(N\) is a RSFA-number} \},
\]
\[
|G:A|^+\geq\max\{|\gamma|: \; \gamma \; \mbox{is a maximal RSFA sequence} \}.
\]

\item Let \(m\) be a given positive integer, then
\begin{itemize}
\item \(|G:A|^-=m+1\)  if and only if
\(A^{-1}Ax_0\cup\cdots\cup A^{-1}Ax_{m-1}\neq G\)
for all RSFA sequences \(\{x_n\}_{n\geq 0}\) and \(A^{-1}Ay_0\cup\cdots\cup A^{-1}Ay_{m}= G\)
for some RSFA sequences \(\{y_n\}_{n\geq 0}\).
\item \(|G:A|^+=m+1\) if and only if
\(A^{-1}Ax_0\cup\cdots\cup A^{-1}Ax_{m}= G\)
for all RSFA sequences \(\{x_n\}_{n\geq 0}\) and \(A^{-1}Ay_0\cup\cdots\cup A^{-1}Ay_{m-1}\neq G\)
for some RSFA sequences \(\{y_n\}_{n\geq 0}\).
\item \(|G:A|_r=m+1\) (i.e., \(A\) is right index stable with the right index \(m+1\)) if and only if
\(A^{-1}Ax_0\cup\cdots\cup A^{-1}Ax_{m}= G\)
and \(A^{-1}Ax_0\cup\cdots\cup A^{-1}Ax_{m-1}\neq G\)
for all RSFA sequences \(\{x_n\}_{n\geq 0}\).
\end{itemize}
\item \(|G:A|^-=\aleph_0\) if and only if
all the RSFA sequences are infinite and some of them maximal.
\end{enumerate}
\end{theorem}
\begin{corollary}\label{RSFAProperties}
Let \(G\) be an arbitrary group, \(A\subseteq G\). Then
\begin{enumerate}[a)]
\item If \(\mathcal{C}(A)\cap \cdots\cap\mathcal{C}(A)^m=\emptyset\), then (all the RSFA-numbers
are less than or equal to \(m\) and so) \(|G:A|^+\leq m\).
\item For the particular cases of ... (\(m=1,2,3\)) we have
\[
|G:A|_r=1\Leftrightarrow \mathcal{C}(A)=\emptyset \Leftrightarrow \text{ all the RSFA-numbers are 1}\Leftrightarrow \Dif(A)=G
\]
\[
|G:A|_r=2\Leftrightarrow \mathcal{C}(A)\neq\emptyset,
\mathcal{C}(A)\cap \mathcal{C}(A)\mathcal{C}(A)=\emptyset \Leftrightarrow
\mathcal{C}(A)\cap \mathcal{C}(A)x=\emptyset
\text{ ; }\forall x\in \mathcal{C}(A)
\]
\[
|G:A|_r=3\Leftrightarrow
\mathcal{C}(A)\cap \mathcal{C}(A)\mathcal{C}(A)\neq\emptyset ,
\mathcal{C}(A)\cap \mathcal{C}(A)x\cap\mathcal{C}(A)y=\emptyset
\text{ ; }\forall x\in \mathcal{C}(A),\ \forall y\in \mathcal{C}(A)\cap \mathcal{C}(A)x
\]
\end{enumerate}
\end{corollary}
\begin{problem}
Find an infinite group \(G\) with a subset \(A\) such that:
\begin{enumerate}[a)]
\item all RSFA sequences related to $A$ are infinite and maximal.
\item some RSFA sequences related to $A$ are finite and some of them are infinite (resp. infinite and maximal).
\item the set of all RSFA numbers related to $A$ is equal to $\mathbb{Z}_+$
(resp. $\mathbb{Z}_+\cup\{\aleph_0\}$). Note that it never can be $\mathbb{Z}_+\cup\{0\}$.
\end{enumerate}
\end{problem}
\begin{question}
What happens for the right subindices of $A$ if all the  RSFA sequences related to $A$ are infinite and maximal?
\end{question}
\begin{remark}\label{RSFA}
For calculating the RSFA numbers and subindices of \(A\) (specially in countable groups),
one can choose \(g_0=1\) without lose of the generality. Also, for choosing \(g_1\)'s (from the symmetric subset \(\mathcal{C}(A)\) if \(g_0=1\)),
when it takes the value \(x\) of \(\mathcal{C}(A)\), it no longer needs to take the value \(x^{-1}\) if
 \(\mathcal{C}(A)\) is a normal subset of \(G\) (i.e.,  \(g\mathcal{C}(A)=\mathcal{C}(A)g\), for all $g\in G$, or equivalently $A$ is a fixed subset
 of every inner automorphism of $G$). For if \(\{g_n\}_{n\geq 0}\) is a RSFA sequence and \(g_n'=g_ng_0^{-1}\), then
 \(\{g_n'\}_{n\geq 0}\) is also a RSFA sequence such that \(g'_0=1\), with the same length as \(\{g_n\}_{n\geq 0}\), and with
 the same maximally status. Now, if \(\{g_n\}_{n\geq 0}\) is a RSFA sequence and \(\mathcal{C}(A)\) is normal,
 then \(\{g_n^{-1}\}_{n\geq 0}\) is also a RSFA sequence with the same length and
 the same maximally status.
 This fact reduces the number of calculations considerably in practice. Indeed, for determining the RSFA numbers
  of \(A\) we only need the lengthes of all sequences \(\{g_n\}_{n\geq 0}\)
 of elements of \(G\) such that:\\
 (1) \(g_0=1\);\\
 (2) \(g_{n+1}\in G\setminus A^{-1}Ag_0\cup \ldots \cup A^{-1}Ag_n(=\mathcal{C}(A)g_0\cap \ldots\cap \mathcal{C}(A)g_n)\);\\
 (3) \(\cup_{n\geq 0} A^{-1}Ag_n=G\),\\
 if exist (and can be with the restriction on \(g_1\) as mentioned above).  Such sequences are enough
for calculating right subindices if \(G\) is countable.
\end{remark}
\begin{example}\label{square}
We know $|\mathbb{Z}:\{0,1,4,9,\ldots\}|=2$ (i.e., the set of all square integers is index stable
with index 2). Because  $\{0,1,4,9,\ldots\}-\{0,1,4,9,\ldots\}=\mathbb{Z}\setminus(4\mathbb{Z}+2)$,
$\mathcal{C}(\{0,1,4,9,\ldots\})=4\mathbb{Z}+2$, and $\mathcal{C}(\{0,1,4,9,\ldots\})+\mathcal{C}(\{0,1,4,9,\ldots\})=4\mathbb{Z}$,
thus the Corollary~\ref{RSFAProperties} implies the claim.

 Now, put $A=\{1,4,9,\ldots\}=\{0,1,4,9,\ldots\}\setminus\{0\}$ (the
set of all positive squares). We know it is non-index stable (see the proof of \cite[Lemma 3.17]{Hooshmand2020}),
but the values of its upper and lower subindices is unknown.
As a very more strong work, we want to calculate $\SubF^1(\mathbb{Z}:\{1,4,9,\ldots\})$. For calculating it we use the RSFA algorithm
for infinite groups as mentioned before. Firs one can check that
$\mathcal{C}(A)=(4\mathbb{Z}+2)\cup \{\pm1 , \pm4\}$ (also see Lemma~\ref{perfect square}). Hence, due to the above remark, we can choose
$g_0=0$ and $g_1\in \{ 1,4, 4t_1+2\}$, for some $t_1\geq 0$, in the RSFA. Thus we have the following cases:\\
\underline{$g_1=1$}. we observe that $g_2\in \mathcal{C}(A)\cap 1+ \mathcal{C}(A)=\{2,-1\}$,
and we get the subfactors $\{0,1,2\}$ and
$\{0,1,-1\}=-1+\{0,1,2\}$.\\
\underline{$g_1=4$}. we have $g_2\in \mathcal{C}(A)\cap 4+ \mathcal{C}(A)=4\mathbb{Z}+2$, and so
$g_2=4t_2+2$, for some integer $t_2$. Since $g_3\in 4\mathbb{Z}+2\cap 4t_2+2+ \mathcal{C}(A)=\{4t_2-2, 4t_2+6\}$,
we have two choices for $g_3$ that give us the subfactors $\{0,4,4t_2+2, 4t_2-2\}$ and $\{0,4,4t_2+2, 4t_2+6\}$.\\
\underline{$g_1=4t_1+2$} ($t_1\geq 0$). in this case
$$g_2\in \mathcal{C}(A)\cap 4t_1+2+ \mathcal{C}(A)=\left\{\begin{array}{ll} \{-2,6,1,-4,4\} &  \; t_1=0\\
\{4t_1-2, 4t_1+6,-4,4\}\; & \;\;\;\; t_1>0
\end{array}
\right.$$
Thus we have the five subcases $g_2=1$, $g_2=4t_1-2\geq -2$, $g_2=4t_1+6\geq 6$,$g_2=-4$, or $g_2=4$  that give us the subfactors as follows:
$$
\{0,1,2\}, \{0,2,4\}, \{0,2,-2,-4\}, \{0,2,-2,4\}, \{0,2,-2,6\}, \{0, 4t_1+2, 4t_1-2,-4\}, \{0, 4t_1+2, 4t_1+6,4\}.
$$

Therefore
$$
\SubF^1(\mathbb{Z}:\{1,4,9,\ldots\})=\Big\{ \{0,1,2\}, \{0,1,-1\},\{0,2,4\}, \{0,2,-2,-4\},
$$

$$
 \{0,2,-2 ,4\},\{0,2,-2,6\}, \{0, 4t_1+2, 4t_1-2,-4\}, \{0, 4t_1+2, 4t_1+6,4\},
$$

$$\{0,4,4t_2+2, 4t_2-2\},\{0,4,4t_2+2, 4t_2+6\}\; : \; t_1\in \mathbb{Z}_+ , t_2\in \mathbb{Z} \Big\}$$

Hence

$|\mathbb{Z}:\{1,4,9,\ldots\}|^-=3$ and $|\mathbb{Z}:\{1,4,9,\ldots\}|^+=4$. Therefore
$\mathrm{pack}_\mathbb{Z}(\{1,4,9,\ldots\})=4$.

As another interesting example, let $G=\mathbb{Z}$ be the additive group of integers and $A=3\mathbb{Z}$.
Then $\mathcal{C}(A)=3\mathbb{Z}+1\cup 3\mathbb{Z}+2$, $\mathcal{C}(A)\cap \mathcal{C}(A)+\mathcal{C}(A)=3\mathbb{Z}+2$,
 and $\mathcal{C}(A)+x$ is equal to $3\mathbb{Z}\cup 3\mathbb{Z}+1$
or $3\mathbb{Z}\cup 3\mathbb{Z}+2$, for every $x\in \mathcal{C}(A)$. Thus $\mathcal{C}(A)\cap \mathcal{C}(A)+x$ is
either $3\mathbb{Z}+1$ or $3\mathbb{Z}+2$, and so $\mathcal{C}(A)\cap \mathcal{C}(A)+x\cap \mathcal{C}(A)+y=\emptyset$,
for all $x\in \mathcal{C}(A)$ and $y\in \mathcal{C}(A)\cap \mathcal{C}(A)+x$. Therefore $|\mathbb{Z}:3\mathbb{Z}|=3$
as we expected.
\end{example}

\subsection{Sub-index properties for countable groups}
Index stability of finite groups and their subsets has been investigated in \cite{Hooshmand2023, HooshmandYousefian}.
In particular, all finite index-stable groups (right, left, and two-sided) have been completely characterized,
and it has been shown that there are only \(14\) such groups. Consequently, in studying index stability beyond the finite case,
attention naturally shifts to countably infinite groups.

The first problem concerning subindices in a countably infinite group is to determine whether a given subindex
is a positive integer or equal to \(\aleph_{0}\). For this class of infinite groups, we obtain the following results,
which are also relevant in the contexts of the integers and the rational numbers.

The next proposition follows immediately from Theorem~\ref{BasicTheorem},
Corollary~\ref{RSFAProperties}, and Corollary~\ref{CountableCor}.

\begin{proposition}\label{Countable groups}
Let $G$ be a countably infinite group and $A$ a given subset of $G$. \\
\begin{enumerate}[a)]
\item If  \(|G:\langle A\rangle|\) is not finite, then \(A\) is index stable and \(|G:A|=\aleph_0\).
\item All sub-factors (relative to \(A\)) can be obtained by the RSFA and so RSFA is a complete algorithm for \(G\).
\item If $\Dif(A)$ is not right syndetic then $A$ is right index stable and $|G:A|_r=|G|=\aleph_0$,
moreover, the union of $A$ with any finite subset $F$ is also
right index stable with the infinite index if  \(f^{-1}Af\subseteq A\) for every \(f\in F\), and if all elements of \(F\)
commute with each other.
\item If there is an infinite subset $B$ such that $A^{-1}A\cap BB^{-1}=\{1\}$, then \(|G:A|^+=\aleph_0\).
\end{enumerate}
\end{proposition}
\begin{corollary}\label{Cor Countable groups}
Let $G$ be countably infinite group such that all proper subgroups are of infinite index (e.g, $\mathbb{Q}$). For
checking the right index stability of $G$ it is enough to only consider generating subsets $A$ such that
$\Dif(A)$ is right syndetic in $G$  but not a subgroup.
\end{corollary}
\begin{example}
The set $\{n!: n=0,1,2,3,\cdots \}$ is index stable with index $\aleph_0$ in $(\mathbb{Z},+)$.
Because its difference set is non-syndetic (also see Theorem~\ref{limit}). Also,
$|\mathbb{Z}:\{n!: n\geq n_0 \}|=|\mathbb{Z}:\{n_0!,(n_0+1)!,(n_0+2)!, \ldots \}|=\aleph_0$, for every fixed integer $n_0\geq 0$.
We do not have similar property for $\{n_0^2,(n_0+1)^2,(n_0+2)^2, \ldots \}$ (as we see later).
\end{example}
We have also an interesting result that says if no finitely many right translates of $\Dif(A)$ covers \(G\),
then there is infinitely many right disjoint translates of $A$ itself. In other words, if $A^{-1}A$ fails to be syndetic, then
A admits an infinite packing by right translates.
\begin{corollary}
Let $A$ be a subset of the countably infinite group $G$. If $\Dif_\ell(A)$ is not right syndetic, then
there is a (maximal) countably infinite subset $B=\{b_1,b_2,\cdots, b_n,\cdots \}$ such that $Ab_1, Ab_2,\cdots, Ab_n,\cdots$ are disjoint.
Hence, every maximal solution $X$ of the equation $\Dif_\ell(A)\cap \Dif_r(X)=\{1\}$ is countably infinite.
\end{corollary}
The following table provides a valuable tool for assessing the (right) index stability of a subset (or sequence) within a countably infinite group. A comparable table for finite groups is presented in~\cite{Hooshmand2023}. The table relies on established criteria for $A$ and $\Dif(A)$, enabling the evaluation of properties such as infiniteness, syndeticity, left factorness, subgroupness, and related characteristics.
It also should be very usefull for studing the index stabilty of subsets and sequences of integer and rational numbers.
\\

\begin{center}
\begin{table}[h]
\begin{tabular}{|p{4.9cm}|p{4.1cm}|p{5.7cm}|}
 \hline
 \multicolumn{3}{|c|}{Main classes of index/right index stable subsets of \textbf{countably infinite groups}} \\
 \hline
 Class Description & index/right index & Explanation/conditions \\
 \hline
 $|A|<|A^c|$ (finite) & $|G:A|=\aleph_0$ & index stable, the converse is not true \\
 \hline
 $|A|>|A^c|$ (co-finite) & $|G:A|=1$ & index stable, the converse is also true \\
 \hline
 $|A|=|A^c|$ (cardinally balanced) &
 $\begin{array}{c}
   |G:A|_r=|G: \Dif(A)| \\
   |G:A|_r=\aleph_0 \\
   |G:A|=\aleph_0
 \end{array}$
 &
 $\begin{array}{c}
 \mbox{if } \Dif(A)\leq G \\
  \mbox{if } \Dif(A) \mbox{ is not right syndetic} \\
  \mbox{if } |G:\langle A\rangle| \mbox{ is infinite}
 \end{array}$ \\
 \hline
\end{tabular}
 \caption{}
  \label{table1}
\end{table}
\end{center}
\subsection{Index pervasive infinite groups}
Motivated by a property of the subindices of subsets of finite groups we lead to some
related concepts for finite and infinite groups.
\begin{definition}
We call an infinite group $G$ right (resp. infinitely right) upper index pervasive if
for every cardinal number $1\leq \kappa\leq |G|$ (resp. $\aleph_0\leq \kappa\leq |G|$) there exists a subset $A$
such that $|G:A|^+=\kappa$ (the concepts ``left (resp. infinitely left) upper/lower index pervasive'',
are analogously defined).
Also, $G$ is called right index (resp. infinitely right index) pervasive if for every $1\leq \kappa\leq |G|$
(resp. $\aleph_0\leq \kappa\leq |G|$) there exists
a right index stable subset with the right index $\kappa$ (the other cases are defined analogously).
\end{definition}
The following theorem is a immediate result of the corollary of the main theorem of \cite{Lyaskovska2007} and Lemma~\ref{covering and packing}.
\begin{theorem}\label{Indexpervasive}
Let \(G\) be an infinite abelian group. Then, for every cardinal number $2,3\neq \kappa \leq |G|$ there is a subset of \(G\)
with the upper index $\kappa$. For the case $\kappa=2$ (resp. $\kappa=3$) also the property is held if and only if
$G$ is not isomorphic to $\bigoplus \mathbb{Z}_3$ (resp. $G$ is not isomorphic to $\bigoplus \mathbb{Z}_2$ or to $\mathbb{Z}_4 \oplus (\bigoplus \mathbb{Z}_2)$.
\end{theorem}
The theorem completely characterizes upper index pervasive infinite abelian groups as follows.
\begin{corollary}\label{...}
An infinite abelian group is \underline{upper index pervasive} if and only if it is not isomorphic to
$\bigoplus \mathbb{Z}_2$, $\bigoplus \mathbb{Z}_3$, and $\mathbb{Z}_4 \oplus (\bigoplus \mathbb{Z}_2)$.
\end{corollary}
\textbf{Note.} At Section 3 of this paper, we prove a stronger property for all subgroups of the additive group of real numbers which says they are
\underline{index pervasive}. This can be extended for all subgroups of $(\mathbb{C},+)$ since it is isomorphic to $(\mathbb{R},+)$.\\
\textbf{Project I.} A vast study of index and subindex pervasiveness of infinite/finite (abelian and nonabelian) groups
should be useful for the behavior of subindices of groups. The final aim is characterization or classification of
some types of index pervasive groups (similar to the above corollary).
\\
Regarding the above project some results such as Proposition~\ref{codif} and the following theorem can be useful.
\begin{theorem}[Complete classification for infinite groups with a right index stable subset with index 2]
An infinite group contains a right index stable subset with the right index 2 if and only if
it contains a non-identity element of order opposite 3.
\end{theorem}

\begin{proof}
Let $G$ be an infinite group an pick a non-identity element with $o(x)\neq 3$. Theorem~\ref{codif} implies there is a subset $A$ such that
$\mathcal C_\ell(A)=\{x,x^{-1}\}$.Thus
$$
B\in\SubF^1_r(G:A)=\left\{\begin{array}{ll} \{\{1,x\}\}; &  \; O(x)=2\\
\{\{1,x\},\{1,x^{-1}\}\}; &  \; O(x)>3
\end{array}
\right.
$$
Thus $A$ is right index stable and $|G:A|_r=2$
(equivalently $|G:A|^+=2$, see Theorem 3.12(f) of \cite{Hooshmand2020}, also it can be obtained by Lemma 3.14 of \cite{Hooshmand2020}).
Conversely, if $|G:A|^+=2$ for some $A\subseteq G$, then $A\{1,x\}=A\cdot \{1,x\}$ for some $x\neq 1$,
and so $A^{-1}A\cap \{\{1,x,x^{-1}\}=\{1\}\}$. Hence, we conclude that $O(x)\neq 3$, since otherwise $\{1,x,x^{-1}\}$
is a subgroup and so $|G:A|^+>2$.
\end{proof}
\begin{remark}
Since in every finite group $G$, all the subindices belongs to $\{1,\cdots , \lfloor \frac{|G|}{2}\rfloor\}\cup\{ |G|\}$,
for finite groups the item $1\leq \kappa\leq |G|$ should be replaced by $\kappa\in \{1,\cdots , \lfloor \frac{|G|}{2}\rfloor\}\cup\{ |G|\}$.
It is obvious that every group of order $\leq 5$ is index pervasive.
One may check that every group of order $6,7$ is also index pervasive.
Hence, we have an example that is index pervasive but not index stable
($\{ |\mathbb{Z}_6:A| : A\;\mbox{is an index stable subset}\}=\{ 1,2,3,6\}$).
Also,
$$
\{ |\mathbb{Z}_8:A| : A\;\mbox{is an index stable subset}\}=\{ 1,2,4,8\}\; , \;
\{ |\mathbb{Z}_8:A|^\pm : A\subseteq G\}=\{ 1,2,3,4,8\}.
$$
Hence it is sub-index but not index pervasive. Note that other groups of order $8$
are index stable but not (any type of) index pervasive (there is no any subset with sub-index $3$).
\end{remark}
\section{Index stability, sub-indices, and $\infty$-difference length of number sets and groups: Interactions to the additive number theory and combinatorics}
One of the most important aspects and applications of the theory of
subindices and subfactors is actually their interactions to the additive number theory and combinatiorics
(e.g., packing numbers, covering numbers, aditive basis, and diameters respect to subsets).
In this regard, we correspond subindices and the length of its infinite difference to any set or sequence of numbers,
 which gives us good information about its behavior related the mentioned aspects.\\
  \textbf{Convention.} Let $\{a_n\}$ be a sequence of elements of a group of numbers. By $|G:\{ a_n\}|^\pm$ we mean
$|G:A|^\pm$ where $A$ is the image of the sequence $a_n$ (i.e., $A=\{a_1,a_2,a_3, \cdots\}$).
Similar conceptions for the sequences can be used according other notations and definitions.\\
\textbf{Index stability of the main additive groups of numbers and subsets.}
Although we showed in response to some open problems in the first section of the paper that every infinite group is not index stable,
 we now consider groups of numbers directly by proving some general theorems and constructing some examples in order to obtain more tools.
We study the index stability of subsets of the groups of numbers, especially integers, rationales, reals, and complex numbers.
The main challenging subject is: determining subindices of arbitrary subsets and their index stability status.\\
It is enough to consider the additive group of real numbers and its subgroups
by the next explanation. \\
Let $G_1,G_2$ be two groups, $f$ a homomorphism from $G_1$ in to $G_2$, and $A,B\subseteq G_1$. It is obvious that
$f(Dif_\ell(A))=Dif_\ell(f(A))$. Hence $G_1=Dif_\ell(A)B$ implies $f(G_1)=Dif_\ell(f(A))f(B)$.
Also, if $f$
is a monomorphism, then $AB=A\cdot B$ if and only if $f(A)f(B)=f(A)\cdot f(B)$. Therefore
$f(B)\in\SubF_r(f(G_1):f(A))$ if and only if $B\in\SubF_r(G_1:A)$. Hence
$|G_1:A|^\pm=|f(G_1):f(A)|^\pm$, and so $A$ (resp.  $G_1$) has
the right index stability  if and only if $f(A)$ (resp. $f(G_1)$) is so. In particular, it implies that two isomorphic groups have the same status of
(right, left, and two-sided) index stability.
This truth convinces us that it is enough to study only index stability of subgroups and subsets of real numbers, since
$(\mathbb{C},+)\cong (\mathbb{R},+)$. In \cite{Hooshmand2020}, it is shown that the set of all square numbers is not index stable in $(\mathbb{Z},+)$.
The next lemma constructs a non-index stable subset of $(\mathbb{Q},+)$ (also see \cite{StackNonindexRationals}).
\begin{lemma}\label{rationals}
Consider the additive group of rationale numbers, fix a positive integer $k$, and put
$$
A=A_k:=\bigcup_{n=-\infty}^{+\infty}\Big[n-\frac{1}{k},n+\frac{1}{k}\Big]_\mathbb{Q},
$$
where $[a,b]_\mathbb{Q}$ denotes the intersection of the closed interval $[a,b]$ and $\mathbb{Q}$.
Then $A_k$ is non-index stable for $k=6,7,8$.
\end{lemma}
\begin{proof}
First, we have
$$
A-A=\bigcup_{n=-\infty}^{+\infty}\Big[ n-\frac{2}{k},n+\frac{2}{k}\Big]_\mathbb{Q}.
$$
Note that $\frac{1}{2}\notin A-A$ if and only if $k\geq 5$, and if it is the case then
$$
(A-A)\cup \frac{1}{2}+(A-A)=\bigcup_{n=-\infty}^{+\infty}\Big[ n-\frac{2}{k},n+\frac{2}{k}\Big]_\mathbb{Q}\cup
\Big[ n+\frac{k-4}{2k},n+\frac{k+4}{2k}\Big]_\mathbb{Q}.
$$
This union is the whole $\mathbb{Q}$ if and only if $k\leq 8$. Therefore $|\mathbb{Q}:A_k|^-=2$ for $k=5,6,7,8$.\\
On the other hand $\frac{5}{14}\notin A-A$ if and only if $k\geq 6$, and if it is the case then
$$
(A-A)\cup \frac{5}{14}+(A-A)=\bigcup_{n=-\infty}^{+\infty}\Big[ n-\frac{2}{k},n+\frac{2}{k}\Big]_\mathbb{Q}\cup
\Big[ n+\frac{5k-28}{14k},n+\frac{5k+28}{14k}\Big]_\mathbb{Q}.
$$
Thus for $6\leq k\leq 10$ we have
$$
(A-A)\cup \frac{5}{14}+(A-A)=\bigcup_{n=-\infty}^{+\infty}\Big[ n-\frac{2}{k},n+\frac{5k+28}{14k}\Big]_\mathbb{Q}\neq \mathbb{Q}.
$$

Hence $|\mathbb{Q}:A_k|^+\geq 3$ for $k=6,7,8,9,10$. Therefore $A_k$ is not index stable in rationales for  $k=6,7,8$.
\end{proof}
\begin{problem}
Calculate $|\mathbb{Q}:A_k|^\pm$ for a given positive integer $k$.
\end{problem}

To examine the index stability of the group of real numbers and subgroups, we need the index stability of the product of a family of groups, which, of course, will also be useful for studying other infinite groups. It is worth noting that \underline{in Remark 3.16 \cite{Hooshmand2020}}, before Lemma 3.17,
\underline{the statement} "$B_1\times B_2\in \mbox{SubF}_r(G_1\times G_2:A_1\times A_2)$ if and only if
$B_1\in\mbox{SubF}_r(G_1:A_1)$ and $B_2\in\mbox{SubF}_r(G_2:A_2)$" \underline{is not correct}.
Hence \cite[Theorem~3.19]{Hooshmand2020} \underline{is not valid in general}. The corrected statement is
mentioned in the proof of the next theorem (for the general case). Also, the following theorems and corollaries improve Theorem 3.19 of  \cite{Hooshmand2020}.
\begin{theorem}\label{product inequality}
Let $\{G_i\}_{i\in I}$ be a family of groups and $A_i\subseteq G_i$ for all $i\in I$. Then
$$
|\prod_{i\in I}G_i:\prod_{i\in I}A_i|^-\leq \prod_{i\in I}|G_i:A_i|^-\leq \prod_{i\in I}|G_i:A_i|^+
$$
Moreover, if $I$ is finite, then
$$
|\prod_{i\in I}G_i:\prod_{i\in I}A_i|^-\leq \prod_{i\in I}|G_i:A_i|^-\leq \prod_{i\in I}|G_i:A_i|^+\leq |\prod_{i\in I}G_i:\prod_{i\in I}A_i|^+
$$
\end{theorem}
\begin{proof}
It can be shown that
$$\Dif_\ell(\prod_{i\in I}A_i)=\prod_{i\in I}\Dif_\ell(A_i),$$
and  so $G_i=\Dif_\ell(A_i)B_i$ where $B_i\subseteq G_i$, for all $i\in I$, implies that
$\prod_{i\in I}G_i=\Dif_\ell(\prod_{i\in I}A_i)(\prod_{i\in I}B_i)$.
Also,
$$
(\prod_{i\in I}A_i)(\prod_{i\in I}B_i)=(\prod_{i\in I}A_i)\cdot (\prod_{i\in I}B_i) \Leftrightarrow A_iB_i=A_i\cdot B_i \mbox{, for all $i\in I$.}
$$
Hence, if  $B_i\in\SubF_r(G_i:A_i)$, for all $i\in I$, then  $\prod_{i\in I}B_i\in\SubF_r(\prod_{i\in I}G_i:\prod_{i\in I}A_i)$.
This requires that the following map is one to one:
$$
\Psi: \prod_{i\in I}\SubF_r(G_i:A_i) \rightarrow \SubF_r(\prod_{i\in I}G_i:\prod_{i\in I}A_i)
$$
$$
\Psi((B_i)_{i\in I})=\prod_{i\in I}B_i
$$
Note that this implies that a copy of the domain of $\Psi$ is a subset of its codomain (not exactly a subset, since a subset of
$\prod_{i\in I}G_i$ does not need to be of the form of a product of subsets). Therefore
$$
|\prod_{i\in I}G_i:\prod_{i\in I}A_i|^-=\inf\{|\mathcal{B}|:\mathcal{B}\in \SubF_r(\prod_{i\in I}G_i:\prod_{i\in I}A_i)\}
\leq \inf\{|\mathcal{B}|:\mathcal{B}\in \Psi(\prod_{i\in I}\SubF_r(G_i:A_i))\}
$$
$$=
\inf\{|\mathcal{B}|:\mathcal{B}\in \{\prod_{i\in I}B_i :  (B_i)_{i\in I}\in \prod_{i\in I}\SubF_r(G_i:A_i)\}\}
$$
$$
= \inf\{|\prod_{i\in I}B_i|:B_i\in \SubF_r(G_i:A_i) \mbox{, for all $i\in I$}\}
= \inf\{\prod_{i\in I}|B_i|:B_i\in \SubF_r(G_i:A_i) \mbox{, for all $i\in I$}\}
$$
$$= \prod_{i\in I}\inf\{|B_i|:B_i\in \SubF_r(G_i:A_i)\}=\prod_{i\in I}|G_i:A_i|^-$$
Now, if $I$ is finite, then
$$
|\prod_{i\in I}G_i:\prod_{i\in I}A_i|^+=\sup\{|\mathcal{B}|:\mathcal{B}\in \SubF_r(\prod_{i\in I}G_i:\prod_{i\in I}A_i)\}
\geq \sup\{|\mathcal{B}|:\mathcal{B}\in \Psi(\prod_{i\in I}\SubF_r(G_i:A_i))\}
$$
$$=
\sup\{|\mathcal{B}|:\mathcal{B}\in \{\prod_{i\in I}B_i :  (B_i)_{i\in I}\in \prod_{i\in I}\SubF_r(G_i:A_i)\}\}
$$
$$
= \sup\{|\prod_{i\in I}B_i|:B_i\in \SubF_r(G_i:A_i) \mbox{, for all $i\in I$}\}
= \sup\{\prod_{i\in I}|B_i|:B_i\in \SubF_r(G_i:A_i) \mbox{, for all $i\in I$}\}
$$
$$= \prod_{i\in I}\sup\{|B_i|:B_i\in \SubF_r(G_i:A_i)\}=\prod_{i\in I}|G_i:A_i|^+$$
It is worth noting that in the above calculations we used the fact that
if $\mathcal{S}_i$ ($i\in I$) is a family of sets whose elements are sets, then
$$\inf\Big\{\prod_{i\in I}|B_i|:B_i\in \mathcal{S}_i \mbox{, for all $i\in I$}\Big\}
=\prod_{i\in I}\inf\{|B_i|:B_i\in \mathcal{S}_i\},$$ but a similar property is valid for the supremum if $I$ is finite.
Now, we can obtain the results, since always $$\prod_{i\in I}|G_i:A_i|^-\leq\prod_{i\in I}|G_i:A_i|^+.$$
\end{proof}
\begin{corollary}
Let $G_1,\ldots,G_n$ be groups and $A_i\subseteq G_i$ for $i=1,2,\ldots,n$. Then
$$
|G_1\times \ldots \times G_n:A_1\times \ldots \times A_n|^-\leq |G_1:A_1|^-\ldots |G_n:A_n|^-\leq |G_1:A_1|^+\ldots |G_n:A_n|^+
$$
$$
\leq |G_1\times \ldots \times G_n:A_1\times \ldots\times A_n|^+
$$
\end{corollary}
The following theorem is one of the most results of the above theorem (compare with
Theorem~\ref{Subgroupsindex}). It is worth noting that here the index set $I$ is arbitrary.
\begin{theorem}\label{family product index stability}
Let $\{G_i\}_{i\in I}$ be a family of groups. If
$\prod_{i\in I}G_i$ is right index stable, then $G_i$ is right index stable for
all $i\in I$.
\end{theorem}
\begin{proof}
If $G_{i_0}$ is not right index stable , for some $i_0\in I$, then there is a non right index stable
subset $\Delta_{i_0}$ of  $G_{i_0}$. For $i\in I$, put
$$
A_i=\left\{\begin{array}{l} \Delta_{i_0} \;\; : \;\; i= i_0    \\
G_i \;\; : \;\; i\neq i_0
\end{array}
\right.
$$
Then
$$
|\prod_{i\in I}G_i:\prod_{i\in I}A_i|^-\leq \prod_{i\in I}|G_i:A_i|^-=|G_{i_0}:\Delta_{i_0}|^-$$ $$
<|G_{i_0}:\Delta_{i_0}|^+=\prod_{i\in I}|G_i:A_i|^+
$$
This requires that $\prod_{i\in I}A_i$ is not right index stable in $\prod_{i\in I}G_i$ (that is a contradiction)
if we have $$|G_{i_0}:\Delta_{i_0}|^+\leq |\prod_{i\in I}G_i:\prod_{i\in I}A_i|^+.$$
But
$$|G_{i_0}:\Delta_{i_0}|^+
= \sup\{\prod_{i\in I}|B_i|:B_i\in \SubF_r(G_i:A_i) \mbox{, for all $i\in I$}\}
$$
$$
= \sup\{|\prod_{i\in I}B_i|:B_i\in \SubF_r(G_i:A_i) \mbox{, for all $i\in I$}\}\leq
|\prod_{i\in I}G_i:\prod_{i\in I}A_i|^+
$$
Note that the last inequality is obtained by the method is used in the proof of the above theorem. Therefore the proof is complete.
\end{proof}
\begin{corollary}\label{product index stability}
If the  direct product $\prod_{i\in I}G_i$ is (right) index stable,
then $|I|\leq 4$, and every $G_i$ is a  group among the 14 mentioned groups in \cite[Theorem~3.1]{HooshmandYousefian}.
\end{corollary}
In the following remark also we improve some results mentioned in \cite[Remark~3.16]{Hooshmand2020}.
\begin{remark}\label{module index stability}
There are some remarks on subfactors and index stability of the additive group of rings and modules.

\begin{enumerate}

    \item If $(G,+)$ is an $R$-module and $r$ be an element of $R$ with the zero annihilator. Then,
    the additive group of $G$ and $rG$ has the same index stability status, since $G\cong rG$. Now, let $A\subseteq G$
    and $r$ is as mentioned. Then, $rA\subseteq rG \leq G$ and we have $B\in\SubF(G:A)$ if and only if $rB\in\SubF(rG:rA)$, and so
    $|rG:rA|^\pm=|G:A|^\pm$ (because every subset of $rG$ is of the form $rB$, for some $B\subseteq G$, and $|rB|=|B|$). Hence by applying this identity and Theorem \ref{Subgroupsindex} (for $H=rG$ and $rA\subseteq H$) we obtain
    \begin{equation}\label{submodule}
    |G:rA|^\pm=|G:rG||rG:rA|^\pm=|G:rG||G:A|^\pm
    \end{equation}
This inequality has some important results (under the above assumptions) that some of them are:
\begin{itemize}
\item if $|G:rG||G:A|^-\neq|G:rG||G:A|^+$, then $rA$ is not index stable in $G$. Hence, if $|G:rG|$
is either finite or less than $|G:A|^-$, then non-index stability of $A$ (in $G$) implies
 non-index stability of $rA$ in $G$.
\item if $rA$ is index stable in $G$ or $|G:rA|^+=|G:rG|$, then $|G:rG||G:A|^-=|G:rG||G:A|^+$,
and this implies $A$ is index stable in $G$ if (moreover)  $|G:rG|$ is finite.
\item if $G$ is countable then $|G:A|^+=\aleph_0$ implies $|G:rA|^+=\aleph_0$
(we provide an interesting example for it in the next example).
\end{itemize}
\item  If $(G,+,\cdot)$ is a ring with identity and $r$ an invertible element of $G$, then the above part implies that
$$|G:rA|^\pm=|r^{-1}G:A|^\pm$$ for every invertible element $r$. Hence, if $rG=G$ (e.g., if $(G,+,\cdot)$ is a field) then
$|G:rA|^\pm=|G:A|^\pm$. Therefore, $|G:rA+x|^\pm=|G:A|^\pm$ for all $x\in G$, and every invertible element $r$ such that $rG=G$.
\end{enumerate}

\end{remark}
\begin{example}
Consider the $\mathbb{Z}$-module $(\mathbb{Z},+)$, fix $m\in \mathbb{Z}\setminus \{0\}$ and let $A\subseteq \mathbb{Z}$.
Then ... requires that
$$
|\mathbb{Z}:mA|^\pm=|m||m\mathbb{Z}:mA|^\pm=|m||\mathbb{Z}:A|^\pm.
$$
Therefore\\
- $mA$ is index stable in $\mathbb{Z}$ if and only if $A$ is index stable in $\mathbb{Z}$, and if it is the
case, then $|\mathbb{Z}:mA|=|m||\mathbb{Z}:A|$.\\
- if $|\mathbb{Z}:mA|^-=\aleph_0$ (equivalently, $|\mathbb{Z}:A|^-=\aleph_0$), then $A$ is index stable and $|\mathbb{Z}:A|=|\mathbb{Z}:mA|=\aleph_0$.\\
- $|\mathbb{Z}:A|^+=\aleph_0$ if and only if $|\mathbb{Z}:mA|^+=\aleph_0$.\\
As a nice example, fix an integer $m>0$. Then due to Lemma~\ref{perfect square} we conclude that
$$|\mathbb{Z}:\{0,m,4m, \cdots, n^2m,\cdots\}|=2m.$$
$$|\mathbb{Z}:\{m,4m, \cdots, n^2m,\cdots\}|^-=3m\; , \; |\mathbb{Z}:\{m,4m, \cdots, n^2m,\cdots\}|^+=4m.$$


Now, consider the additive group of the field $(\mathbb{Q},+,\cdot)$ and put $A=A_k$ where $A_k$ is as mentioned in Lemma~\ref{rationals}. Then
$rA_k+x$ is non-index stable for all $r\neq 0$, $x\in \mathbb{Q}$, and  $k=6,7,8$ (hence, there are infinitely
many non index stable subsets of rationales).
  \end{example}
\subsection{Sub-indices, index stability, and $\infty$-difference length of subsets of integer numbers}
One of the most interesting and challenging aspects of this topic is the study of subindices and index stability for the integers, which we will pursue here. Because the additive group of integers is countably infinite, we can leverage previous results to establish the following multi-part criterion for them.
\begin{theorem}[Integers index stability Criteria]\label{integer criteria}
All the following conditions are necessary for $A$ to be non-index stable.\\
$($a$)$ $A$ and $A^c$ are infinite,\\
$($b$)$ $A-A$ is not a subgroup,\\
$($c$)$  there is a finite RSFA-number of $A$,\\
$($d$)$ $A-A$ is \textbf{syndetic} in $\mathbb{Z}$,\\
$($e$)$ the equation $(A-A)\cap (X-X)=\{0\}$ has a finite maximal solution.\\
Hence, if a subset $A$ of integers does not satisfy one of the conditions, then $A$ is index stable $($ in $\mathbb{Z})$.
\end{theorem}
\begin{proof}
This is a direct result of Theorem~\ref{BasicTheorem}, Theorem~\ref{InfiniteSubindexEquation}, Corollary~\ref{RSFAProperties},
and Proposition~\ref{CountableCor}.
\end{proof}
The following lemma introduces an important class of index stable subsets of integers
that also shows that the additive group of integers is index pervasive (also see Theorem~\ref{pervasive}).
\begin{lemma}\label{pervasive integers}
If $A$ is a subset of the additive group of integers such that
$C^0(A)=\{0,\pm1,\pm2,\dots,\pm n\}$, for some $n\geq 0$, then
$$
\SubF^0(\mathbb{Z}:A)=\{[-x,y]_\mathbb{Z}: x,y\geq 0\; , \; x+y=n\},
$$
and so $A$ is index stable with $|\mathbb{Z}:A|=n+1$.\\
Also, we have $B-B=C^0(A)$ for all $B\in \SubF^0(\mathbb{Z}:A)$ (a typical such element $B$ is
$\{0,1,2,\dots, n\}$, another one is  $\{0,\pm1,\dots,\pm\frac{n}{2}\}$ if $n$ is even).\\
\end{lemma}
\begin{proof}
Let $B\in \SubF^0(\mathbb{Z}:A)$ and put $x=-\min B$, $y=\max B$. Then, $x,y\geq 0$ (because $0\in B$) and
$[-x,y]_\mathbb{Z}\subseteq [-n,n]_\mathbb{Z}$, where $[a,b]_\mathbb{Z}$ denotes $[a,b]\cap \mathbb{Z}$.
Since $\pm (x+y)\in B-B\subseteq [-n,n]_\mathbb{Z}$, we conclude that
$$[-(x+y),x+y]_\mathbb{Z}=[-x,y]_\mathbb{Z}-[-x,y]_\mathbb{Z}\subseteq [-n,n]_\mathbb{Z}.$$
Now, if $x+y<n$, then
$$B\subsetneqq [-x,y+1]_\mathbb{Z}\subseteq [-x,y+1]_\mathbb{Z}-[-x,y+1]_\mathbb{Z} =[-(x+y+1),x+y+1]_\mathbb{Z}\subseteq [-n,n]_\mathbb{Z},$$
that is a contradiction due to the maximality of $B$. Finally, note that $|B|=n+1$ and $B-B=[-n,n]_\mathbb{Z}$.
\end{proof}
It is worth noting that the syndeticity of the difference sets of integers (part (d) of the above theorem) plays a significant and applicable role in the study of index stability of integer sets. In this regard, the following theorems and corollaries provide important applicable tools.
\begin{lemma}\label{syndetic}
A subset of integers is syndetic $($in $\mathbb{Z})$ if and only if it is two-sided unbounded with bounded gaps.
\end{lemma}
\begin{proof}
Let $S\subseteq \mathbb{Z}$ be syndetic. Then, there is a finite nonempty subset $F$ of positive integers such that $S+F=\mathbb{Z}$
(because $F$ is finite, and $S+X=\mathbb{Z}$ implies $S+(X+m)=\mathbb{Z}$ for all $m\in \mathbb{Z}$). Put
$g=\max F$ and suppose that the set of gaps of $S$ is not bounded. Then there are two consecutive elements $s'$ and $s$
of $S$ such that $s-s'>g$. Since $d\in S+F$, there are $s''\in S$ and $f\in F$ such that $s=s''+f$
(it is obvious that $s''\neq s'$). If $s''>s'$ (resp. $s''<s'$),
then $s''\geq s$ and so $f=s-s''\leq 0$ (resp. $s=s''+f<s'+g<s$) that is a contradiction. Note that if
$S$ is bounded up or below then $S+F$ is so, and this is also a contradiction.\\
Conversely, Let $S$ be two-sided unbounded with bounded gaps by $g$. Then, one can
sort the elements of $S$ as $S=\{\cdots, s_{-2}, s_{-1}, s_{0}, s_{1}, s_{2}, \cdots  \}$.
Hence, if $x\in \mathbb{Z}$ then $s_k\leq x < s_{k+1}$ for a unique $k\in \mathbb{Z}$.
Thus
$$
x\in \{s_{k}, s_{k}+1, \cdots, s_{k+1}-1 \}\subseteq \{s_{k}, s_{k}+1, \cdots, s_{k}+g-1 \}\subseteq
S\cup S \cup \cdots \cup S+g-1=S+\{0, 1, \cdots, g-1\}.
$$
Therefore $S+\{0, \cdots, g-1\}=\mathbb{Z}$ and the proof is complete.
\end{proof}
It is interesting that for checking the syndetic status of $A-A$ we do not need the two-sided unbounded condition.
For if $A$ is finite, then $A-A$ is not syndetic, and if $A$ is infinite, then $A-A$ is two-sided unbounded as mentioned bellow.
\begin{corollary}\label{A-A Syndetic}
Let $A$ be an infinite  subset of integers. Then, the equation $(A-A)+X=\mathbb{Z}$
has a finite solution $($i.e., $A-A$ is syndetic$)$ if and only if $A-A$ has bounded gaps.
\end{corollary}
\begin{proof}
Without loss of the generality assume that $0\in A$ (because $(A-a_0)-(A-a_0)=A-A$ and $0\in A-a_0$ for all $a_0\in A$).
Hence, $A\cup (-A)\subseteq A-A$ and so $A-A$ is two-sided unbounded (since $A$ is infinite and $A\cup (-A)$ is symmetric).
\end{proof}
\begin{corollary}\label{integer aleph_0 index}
If $A-A$ has \underline{unbounded} gaps, then $A$ is index stable with $|\mathbb{Z}:A|=\aleph_0$.
\end{corollary}
\begin{proof}
This is a direct result of Theorem \ref{integer criteria} and Corollary \ref{A-A Syndetic}.
\end{proof}
Although the above corollaries are very important and useful for the study of index stability and subindices of
integer numbers, but the next theorem provides a simple criteria for index stability with the infinite index.
\begin{theorem}\label{limit}
If $\{a_n\}_{n=1}^{\infty}$ is a strictly increasing sequence of integers such that $a_{n}-2a_{n-1}\to \infty$ as $\to \infty$, then  $|\mathbb{Z}:\{a_n\}|=\aleph_0$.
\end{theorem}
\begin{proof}
Denote the range of $\{a_n\}_{n=1}^{\infty}$ by $A$, and
put $D=A-A$, $c_n=a_{n}-2a_{n-1}-1$. Then we have $a_{n}>2a_{n-1}+c_n$,
$c_n\to \infty$, and so there is a positive integer $n_1$ such $c_n>1$ and $a_n>0$ for all $n\geq n_1$. We claim that
$$
[a_{n-1}+1,a_{n-1}+c_n]\cap D=\emptyset\; ; \; n\geq n_1.
$$
This clearly implies that $D$ has unbounded gaps and so $|\mathbb{Z}:A|=\aleph_0$ by  Corollary \ref{integer aleph_0 index}.\\
For proving the claim note that the elements of $D$ are of the form $a_i-a_j$. If $i\leq j$, then $a_i-a_j\leq 0$ and
does not belong to $[a_{n-1}+1,a_{n-1}+c_n]$. Also, if $i>j$, then consider the following cases:\\
(1) $n_1\leq i<n$. We have $a_i<a_j\leq a_{n-1}$ and so $a_i-a_j\leq a_{n-1}$ which requires $a_i-a_j\notin [a_{n-1}+1,a_{n-1}+c_n]$.\\
(2) $i\geq n>j\geq n_1$. We have
$$
a_i-a_j\geq a_n-a_j\geq a_n-a_{n-1}>a_{n-1}+c_n \Rightarrow a_i-a_j\notin [a_{n-1}+1,a_{n-1}+c_n].
$$
(3) $i>j\geq n\geq n_1$. We have
$$
a_i-a_j> 2a_{i-1}+c_i-a_j= (a_{i-1}-a_j)+a_{i-1}+c_i\geq a_{i-1}+c_i\geq a_{n}+c_i
> 2a_{n-1}+c_n+c_i> a_{n-1}+c_n
$$
$$
 \Rightarrow a_i-a_j\notin [a_{n-1}+1,a_{n-1}+c_n].
$$
\end{proof}
\begin{corollary}\label{limitcor}
If $\{a_n\}_{n=1}^{\infty}$ is a strictly decreasing sequence of integers such that $a_{n}-2a_{n-1}\to -\infty$ as $\to \infty$, then  $|\mathbb{Z}:\{a_n\}|=\aleph_0$.
\end{corollary}
\begin{proof}
Note that if $\{a_n\}_{n=1}^{\infty}$ is strictly decreasing and $a_{n}-2a_{n-1}\to -\infty$, then $-a_n$ satisfies the
conditions of \ref{limit} and so
$$|\mathbb{Z}:\{a_n\}|=|\mathbb{Z}:-\{a_n\}|=|\mathbb{Z}:\{-a_n\}|=\aleph_0.$$
\end{proof}
\begin{example}\label{a power example}
Let $a\geq -1$ be a fixed integer and consider the sequence $a_n=a^n$, then   $a_n - 2a_{n-1}=(a-2)a^{n-1}$. Thus, if $a\geq 3$,
then $a_n$ is index stable and $|\mathbb{Z}:\{a^n\}|=\aleph_0$ by Theorem~\ref{limit}. Also, it has the same status if $-1\leq a\leq 1$ since
the range of $a_n$ is finite. For the case $a=2$, we also claim that
the set $A=\{2^n : n \ge 0\}$ is index stable with index $\aleph_0$, because its difference set $\Dif(A) = A-A= \{2^i - 2^j : i, j \ge 0\}$ is not syndetic.
For any integer $n \ge 2$, put
$$ J_n = \left( 2^n, \quad 3 \cdot 2^{n-1} \right). $$
The length of this interval is $ 2^{n-1} - 1$. Since $n$ is arbitrary, this length is unbounded.
We show that $J_n$ contains no element of $\Dif(A)$. Assume $d \in J_n \cap \Dif(A)$, so $d = 2^k - 2^m$ with $k > m \ge 0$,
and hence $d\ge 2^{m+1}-2^m=2^m$. Now, if $m\geq n+1$, then $d\ge 2^{n+1}>3\cdot 2^{n-1}$ that is a contradiction.
Thus, $m\leq n$ and we have the following cases:
\begin{enumerate}
    \item If $k \le n$, then $d \le 2^n - 2^0 = 2^n - 1 < 2^n$. Contradicts $d > 2^n$.
    \item If $k = n+1$, then $d = 2^{n+1} - 2^m< 3 \cdot 2^{n-1}$, and so $2^{n-1} < 2^m$. This implies
    $n=m$ (since $m\leq n$). Hence, $d = 2^{n+1} - 2^n = 2^n$ that is a contradiction.
     \item If $k \ge n+2$, then $d \ge 2^{n+2} - 2^{m}\ge 2^{n+2} - 2^{n}=3 \cdot 2^{n}$ that is impossible.
\end{enumerate}
It is also a good problem to check the index stability of $a^n$ for $a\leq -2$ (it seems the status is similar to the case $a>-2$).\\
Some other examples of index stable sequences with the infinity index which are only obtained by Theorem~\ref{limit} and Corollary~\ref{limitcor}
are as follows: $n!$ (OEIS A000142; see oeis.org), $n^n$ (OEIS A000312), $n^\alpha a^n$ with constant integers $\alpha>0$, $a>1$, Bell numbers (OEIS A000110),
Catalan numbers	(OEIS A000108), Partition numbers (OEIS A000041), $2^n-3^n$, etc.
\end{example}
In continuation, we focus on the subindices, index stability and the $\infty$-difference length of some important sets of integers.
For infinite difference length of subsets, note that it is introduced in \cite{Hooshmand2020} and more studies in
finite groups have been done in \cite{Hooshmand2023}. Also we mentioned their relations to the diameter
concept and covering numbers in the previous section. In particular, for the additive group of integers,
some explanations and examples are mentioned before. Hence, we use the results for more next studies.    \\
\textbf{Index stability and $\infty$-difference length of primes.}
In order to study the sub-indices of primes, there  is a famous conjecture
 (Maillet's Conjecture) about prime numbers that says:
 every even number is the difference of two primes (e.g., see \cite{Guy}).
It is clear that the conjecture is equivalent to $\Dif(\mathbb{P}_o)=\mathbb{Z}_e$, where $\mathbb{P}$
is the set of primes and
$\mathbb{P}_o=\mathbb{P}\setminus \{2\}=\mathbb{P}^*$ (the set of  odd primes). We want to show that it is also equivalent to
the property that the set of primes (and also $\mathbb{P}_o$)  is index stable with index 2. Note that Lemma 3.14 of \cite{Hooshmand2020} say that  for $A\subseteq \mathbb{Z}$
\begin{equation}\label{index two}
|\mathbb{Z}:A|=2\Leftrightarrow \mathcal{C}(A)+\mathcal{C}(A)\subseteq \Dif(A)
\Leftrightarrow (\mathcal{C}(A)+\mathcal{C}(A))\cap \mathcal{C}(A)=\emptyset
\end{equation}
 $$\Leftrightarrow
 \mathcal{C}(A)\cap (\mathcal{C}(A)+x)=\emptyset,
\mbox{ for all } x\in \mathcal{C}(A) \Leftrightarrow
\Dif(A)+\{0,x\}=\mathbb{Z}, \mbox{ for all } x\in \mathcal{C}(A)$$
This has an interesting interpretation that says the subset $A$ is index stable with index two if and only if
$\mathcal{C}(A)$ is sum-free (see \cite{TaoVu}).
By introducing a related hypothesis, and a conjecture we arrive
at the following important theorem. Note that the parts (a) and (d) are equivalent by Theorem 3.30 of \cite{Hooshmand2020}
It is worth \underline{noting} that the \underline{part (b) of \cite[Theorem 3.30]{Hooshmand2023} is wrong} and should be omitted from the theorem, because
$\mathbb{P}^*$ is not a subset of $\mathbb{Z}_e$.
\begin{theorem}\label{prime equivalence}
The following assertions are equivalent:\\
{\bf $($a$)$} The Maillet's Conjecture;\\
{\bf $($b$)$} $|\mathbb{Z}:\mathbb{P}|=2$;\\
{\bf $($c$)$} $|\mathbb{Z}:\mathbb{P}_o|=2$;\\
{\bf $($d$)$} $($Our primes hypothesis$)$ For every (distinct non-zero) integers $a, b$, at least one of the numbers
 $a, b$, and $a-b$ can be expressed as the difference of two primes;\\
{\bf $($e$)$} $($Our primes conjecture$)$ If $a,b$ are two integers such that can not be expressed as the difference of two primes,
then $a+b$ is a difference of two primes;\\
{\bf $($f$)$} $|\mathbb{Z}:\mathbb{P}|^+=2$;\\
{\bf $($g$)$} $|\mathbb{Z}:\mathbb{P}_o|^+=2$.\\
\end{theorem}
\begin{proof}
The Maillet's Conjecture implies (b), since for every triples integers $a,b,a-b$ at least one of them is even.
Hence, let our primes hypothesis is satisfied and fix an arbitrary positive even integer $n$.
Motivated by \cite{StackPrimeConjecture} we claim that there is an odd integer $b$ such that both (odd numbers)
$b$, $b+n$ are not the difference of two primes, or equivalently $b+2$ and $b+n+2$ are not primes.
If the claim is not true, we conclude that $\pi(x)\geq x/8-n/4$ for all real numbers $x$ from a number on
($\pi$ is the prime-counting function) that is impossible (see \cite{Hooshmand2020} for more details). Therefore, (a) and (d) are equivalent.
Also (d) is equivalent to (e), clearly.
On the other hand (b) and (e) are equivalent by \cite[Theorem 3.30]{Hooshmand2023} and the fact that (e) indeed says
$\mathcal{C}(\mathbb{P})+\mathcal{C}(\mathbb{P})\subseteq \Dif(\mathbb{P})$. Now, we show that
(c) is equivalent to the equivalent conditions (a), (b), (d), and (e). Hence, let $\Dif(\mathbb{P}_o)\neq \mathbb{Z}_e$
then $\mathcal{C}(\mathbb{P}_o)=\mathbb{Z}_e\cup S$ for some $\emptyset\neq S\subseteq \mathbb{Z}_o$.
Thus
$$
\mathcal{C}(\mathbb{P}_o)+\mathcal{C}(\mathbb{P}_o)=(\mathbb{Z}_o+\mathbb{Z}_o)\cup (\mathbb{Z}_o+S)\cup  (S+S)=\mathbb{Z}_e\cup \mathbb{Z}_o
=\mathbb{Z},
$$
which implies $|\mathbb{Z}:\mathbb{P}_o|\neq 2$ (by (\ref{index two})). This means (c) implies (a). Conversely, let
(c) does not hold but the others are satisfied. Then $|\mathbb{Z}:\mathbb{P}_o|\geq 3$ (obviously), and so
there are distinct integers $a,b$ such that  $\Dif(\mathbb{P}_o)\cap \{0,a,b,a-b\}=\{ 0\}$ that is impossible,
because $\Dif(\mathbb{P})\cap \{a,b,a-b\}\neq \emptyset$ (by (d)), $\Dif(\mathbb{P})=\Dif(\mathbb{P}_o)\cup (\mathbb{P}_o-2)\cup (2-\mathbb{P}_o)$,
 and at lease one of numbers $a,b$, and $a-b$ is even.\\
 So far we proved that (a), (b), (c), (d), and (e) are equivalent. Now the proof is complete since
 $|\mathbb{Z}:\mathbb{P}|^+=2$ (resp. $|\mathbb{Z}:\mathbb{P}_o|^+=2$) is equivalent to
  $|\mathbb{Z}:\mathbb{P}|=2$ (resp. $|\mathbb{Z}:\mathbb{P}_o|=2$), due to $|\mathbb{Z}:\mathbb{P}|^-\geq2$
  (resp. $|\mathbb{Z}:\mathbb{P}_o|^-\geq2$).
\end{proof}
Now, by using a result of \cite{prime difference} we prove a weak version of the equivalent conditions of the above theorem.
\begin{theorem}
The subindices of the prime numbers are finite and we have $2\leq \mathbb{Z}:\mathbb{P}|^- \leq|\mathbb{Z}:\mathbb{P}|^+\leq 720$.\\
(This bound cane be improved to $2\leq \mathbb{Z}:\mathbb{P}|^- \leq|\mathbb{Z}:\mathbb{P}|^+\leq 18$ if
we assume that the primes have level of distribution $\theta$ for every $\theta < 1$, as mentioned in \cite{prime difference}).
\end{theorem}
\begin{proof}
Denote by $D$ the set of even numbers that can be expressed in infinitely many
ways as the difference of two primes  and put
$\Delta(B)=(B-B)\cap \mathbb{N}$ (as mentioned in \cite{prime difference}). Then, using
Theorem 1.6 of \cite{prime difference} we obtain
$$
\big((\mathbb{P}-\mathbb{P})\cap (B-B)\big)\setminus \{0\}\supseteq \big((\mathbb{P}-\mathbb{P})\cap \mathbb{N}\big)\cap \Delta(B)
\supset D\cap \Delta(B)\neq \emptyset,
$$
for all subsets $B$ with $|B|\geq 721$. Thus $B\notin \SubF(\mathbb{Z}:\mathbb{P})$ for all subsets $B$ such that $|B|\geq 721$
and so $|\mathbb{Z}:\mathbb{P}|^+\leq 720$. Now the proof is complete, since $\mathbb{P}-\mathbb{P}\neq \mathbb{Z}$ or equivalently
$|\mathbb{Z}:\mathbb{P}|^-\geq2$.
\end{proof}
The above theorem has an interesting result as follows.
\begin{corollary}
The set of prime numbers is not a factor of integers. As a stronger result, it is also not a subfactor of integers.
\end{corollary}
\begin{proof}
If $\mathbb{P}+B=\mathbb{P}\dot{+}B$ for some $B\subseteq \mathbb{Z}$, then  $B\subseteq B'$ for some
$B'\in \SubF(\mathbb{Z}:\mathbb{P})$ (by Corollary 3.3 of \cite{Hooshmand2020}) and so the above theorem implies
$|B|\leq |B'|\leq 720$. Thus, if $\mathbb{Z}=\mathbb{P}\dot{+}B$ for some $B\subseteq \mathbb{Z}$, then
  $B$ is finite which requires $\mathbb{P}$ is syndetic, a contradiction. Also, if $\mathbb{P}$ is
a subfactor, then $A+\mathbb{P}=A\dot{+}\mathbb{P}$  and $(A-A)+\mathbb{P}=\mathbb{Z}$ , for some subset $A$
that is impossible since $A$ and so $A-A$ must be finite.
\end{proof}
Another interesting property of prime numbers set in view of our topic is $Dif^2(\mathbb{P})=\mathbb{Z}$.
It is equivalent to the property that for every
$N\in \mathbb{Z}$ there exist $p_1,p_2,p_3,p_4\in \mathbb{P}$ such that $N=p_1+p_2-p_3-p_4$.
\begin{theorem}[{\cite[Theorem 3.31]{Hooshmand2020}}]\label{prime difference lentgh}
The prime numbers set is a generating set of integers with $\infty$-difference length $2$.
Hence $\mathbb{Z}=\langle \mathbb{P}\setminus \{2\} \rangle^{(2)}=\langle \mathbb{P}\rangle^{(2)}$.
\end{theorem}

The next conjecture  is weaker than the Maillet's conjecture due to Theorem~\ref{prime equivalence}.\\
\textbf{Conjecture I.} The set of primes is index stable (in integers).\\
\textbf{Index stability and $\infty$-difference length of perfect squares.}
In view of index stability, the sequences of perfect squares (or higher powers) starting from a given number are particularly interesting and challenging to analyze. Moreover, their difference sets and the $\infty$-difference length of such sequences are noteworthy.
\begin{lemma}\label{perfect square}
For every integer $n_0\geq 0$ we have
\begin{equation}\label{perfectcomplement}
\mathcal{C}(\{n_0^2, (n_0+1)^2, \cdots\})=\left\{\begin{array}{ll} 4\mathbb{Z}+2  &  \; n_0=0\\
(4\mathbb{Z}+2)\cup \{\pm1,\pm3,\cdots,\pm(2n_0-1)\}\cup \{\pm4,\pm8,\cdots,\pm4n_0\}\; & \;\;\;\; n_0>0
\end{array}
\right.
\end{equation}
More over, $\{0,1,4,9, \cdots\}$ is index stable with index $2$, but $\{1,4,9, \cdots\}$ is not index stable and
 $|\mathbb{Z}:\{1,4,9, \cdots\}|^-=3$, $|\mathbb{Z}:\{1,4,9, \cdots\}|^+=4$ (this is an improvement of an statement
 in \cite{Hooshmand2020}$(3.4(c))$.
 \end{lemma}
\begin{proof}
Put $\mathcal{A}_{n_0}=\{n_0^2, (n_0+1)^2, \cdots\}$. First note that
$$
\mathcal{C}(\mathcal{A}_{n_0+1})=\mathcal{C}(\mathcal{A}_{n_0})\dot{\cup}(\Dif(\mathcal{A}_{n_0})\cap \mathcal{C}(\mathcal{A}_{n_0+1}))
$$
Hence
$$
x\in \mathcal{C}(\mathcal{A}_{n_0+1})\Leftrightarrow x\in \mathcal{C}(\mathcal{A}_{n_0})  \vee ( x=q^2-r^2 \mbox{ , for some } q,r\geq n_0\;
\wedge \; x\neq m^2-n^2 \mbox{ , for all } m,n\geq n_0+1)
$$
$$
\Leftrightarrow x\in \mathcal{C}(\mathcal{A}_{n_0})  \vee ( x=\pm(k^2-n_0^2) \neq m^2-n^2\mbox{ , for some } k\geq n_0+1
\mbox{ , all } m,n\geq n_0+1)
$$
But, the identities
$$
k^2-n_0^2=(\frac{k^2-n_0^2+1}{2})^2-(\frac{k^2-n_0^2-1}{2})^2=(\frac{k^2-n_0^2}{4}+1)^2-(\frac{k^2-n_0^2}{4}-1)^2
$$
imply that the above equivalent conditions hold if and only if $k\in \{n_0+1,n_0+2\}$, and so
$$
\mathcal{C}(\mathcal{A}_{n_0+1})=\mathcal{C}(\mathcal{A}_{n_0})\dot{\cup} \{\pm(2n_0+1),\pm(4n_0+4)\}
$$
Now this identity together with $\mathcal{C}(\mathcal{A}_{0})=4\mathbb{Z}+2$ require (\ref{perfectcomplement}) by the induction.\\
We know $|\mathbb{Z}:\mathcal{A}_0|=2$, and
$\mathcal{C}(\mathcal{A}_{1})=(4\mathbb{Z}+2)\cup \{\pm1 , \pm4\}$. Hence, due to Remark~\ref{RSFA} we can chose
$g_0=0$ and $g_1\in \{ 1,4, 4t_1+2\}$, for some $t_1\geq 0$, in the RSFA. Thus we have the following cases:\\
\underline{$g_1=1$}. we observe that $g_2\in \mathcal{C}(\mathcal{A}_{1})\cap 1+ \mathcal{C}(\mathcal{A}_{1})_{1}=\{2,-1\}$, and we get the subfactors $\{0,1,2\}$ and
$\{0,1,-1\}=-1+\{0,1,2\}$.\\
\underline{$g_1=4$}. we have $g_2\in \mathcal{C}(\mathcal{A}_{1})\cap 4+ \mathcal{C}(\mathcal{A}_{1})=4\mathbb{Z}+2$, and so
$g_2=4t_2+2$, for some integer $t_2$. Since $g_3\in 4\mathbb{Z}+2\cap 4t_2+2+ \mathcal{C}(\mathcal{A}_{1})=\{4t_2-2, 4t_2+6\}$,
we have two choices for $g_3$ that give us the subfactors $\{0,4,4t_2+2, 4t_2-2\}$ and $\{0,4,4t_2+2, 4t_2+6\}$.\\
\underline{$g_1=4t_1+2$} ($t_1\geq 0$). in this case
$$g_2\in \mathcal{C}(\mathcal{A}_{1})\cap 4t_1+2+ \mathcal{C}(\mathcal{A}_{1})=\left\{\begin{array}{ll} \{-2,6,1,-4,4\} &  \; t_1=0\\
\{4t_1-2, 4t_1+6,-4,4\}\; & \;\;\;\; t_1>0
\end{array}
\right.$$

Thus we have the five subcases $g_2=1$, $g_2=4t_1-2\geq -2$, $g_2=4t_1+6\geq 6$,$g_2=-4$, or $g_2=4$  that give us the subfactors as follows:
$$
\{0,1,2\}, \{0,2,4\}, \{0,2,-2,-4\}, \{0,2,-2,4\}, \{0,2,-2,6\}, \{0, 4t_1+2, 4t_1-2,-4\}, \{0, 4t_1+2, 4t_1+6,4\}.
$$

Therefore
$$
\SubF^1(\mathbb{Z}:\{1,4,9,\ldots\})=\Big\{ \{0,1,2\}, \{0,1,-1\},\{0,2,4\}, \{0,2,-2,-4\},
$$
$$
 \{0,2,-2,4\},\{0,2,-2,6\}, \{0, 4t_1+2, 4t_1-2,-4\}, \{0, 4t_1+2, 4t_1+6,4\},
$$
$$\{0,4,4t_2+2, 4t_2-2\},\{0,4,4t_2+2, 4t_2+6\}\; : \; t_1\in \mathbb{Z}_+ , t_2\in \mathbb{Z} \Big\}
$$
Hence the proof is complete.
\end{proof}
\begin{problem}
Calculate the $n$-differences ($\Dif^n$), $\infty$-difference length, and subindices of $\mathcal{A}_{n_0}$,
for a given non-negative integer $n_0$.
\end{problem}
\textbf{Index stability and $\infty$-difference length of $A(k)$ and $A[k]$.}
The sequences considered in this part (two forms of powers) are highly significant and challenging, exhibiting various states of index stability and values of infinite difference lengths. By utilizing previously obtained results, we can determine some of their related states in the following lemma, while others remain as unsolved problems.
\begin{lemma}\label{Waringlemma}
For a fixed positive integer \(k\), put
\[
A(k)=\{m^k:m\in \mathbb{Z}^+\cup\{0\}\}\; , \; A[k]=\{k^m:m\in \mathbb{Z}^+\cup\{0\}\}.
\]
Denote by \(g(k)\) the minimum number of \(k\)th powers of naturals needed to represent all positive integers. Then
\begin{enumerate}[(a)]
  \item The \(\infty\)-difference length of \(A(k)\) is finite and \(dl^\infty (A(k))\leq g(k)\).
    \item \(A(1)\) and \(A(2)\) are index stable with indices one and two, respectively. But \(A(2)\setminus \{0\}\) is not index stable,
    and we have $|\mathbb{Z}:A(2)\setminus \{0\}|^-=3$, $|\mathbb{Z}:A(2)\setminus \{0\}|^+=4$.
     \item The \(\infty\)-difference length of \(A[2]\) is infinite.
      \end{enumerate}
\end{lemma}
\begin{proof}
(a) Then Waring's problem (see \cite{WaringBook}) implies that \(\infty\)-difference length of \(A(k)\) is finite and \(dl^\infty (A(k))\leq g(k)\).
Note that \(dl^\infty (A(1))=1=g(1)=1\) but \(dl^\infty (A(2))=2<g(2)=4\), since  \(\Dif(A(2))=\mathbb{Z}_o\cup 4\mathbb{Z}\) and so \(\Dif^2(A(2))=\Dif^3(A(2))=\Dif^\infty(A(2))=\mathbb{Z}\)  (we arrive at a related open problem later).\\
(b), (c) By Lemma~\ref{perfect square} and Example~\ref{WaringExample}, respectively.
\end{proof}
Followed by the above lemma we arrived at the next problems.\\
\begin{problem}
Let $k$ be a fixed positive integer and $A(k)$, $A(k)$ as mentioned above.
\begin{enumerate}[(i)]
\item
Calculate \(\infty\)-difference length of \(A(k)\) (that is finite),  upper and lower subindices of \(A(k)\), and determine its index stability.
\item

  Adapted from \cite{stackwaring2022}, we claim that \(dl^\infty (A[k])=\infty\), for all \(k\geq 2\),
and so \(\mathbb{Z}=\langle A[k] \rangle^{(\infty)}\) (note that \(dl^\infty (A[1])=1\)). Is it true?
 \item
 Prove or disprove: \(A[k]\) is index stable with the infinite index (\(=\aleph_0\)). This result can be extended for all fixed integers \(k\).

\end{enumerate}
\end{problem}
\textbf{Index stability and $\infty$-difference length of Fibonacci numbers.}

Let $\{F_n\}_{n=1}^\infty $ be the Fibonacci sequence (defined by
$ F_n=F_{n-1}+F_{n-2}$ with $F_1=F_2=1$), and denote by $F$ the set of all Fibonacci numbers.
It does not satisfy the conditions of Theorem~\ref{limit} or Corollary~\ref{limitcor}. Because, $F_n$ is strictly increasing for
$n\geq2$, and
$F_n-2F_{n-1}=-F_{n-3}$ for all $n\geq 3$ (and so it tends to $-\infty$). But, however $F-F$ is not syndetic
as one see in the proof of the next theorem.
\begin{theorem}
The Fibonacci sequence is index stable and $|\mathbb{Z}:\{F_n\}|=\aleph_0$.
\end{theorem}
\begin{proof}
We show that the set \(F - F = \{F_m - F_n : m, n \in \mathbb{N}\}\) is not syndetic in \(\mathbb{Z}\). Indeed, for any given \(L\), there exists an interval of length \(L\) containing no element of \(F - F\).
 Choose \(k\) large enough so that \(F_{k-3} > L\), and consider the interval \((F_k, F_k + L]\).
We verify that no difference \(F_m - F_n\) with \(m > n\) lies in this interval.

Case 1: \(m \leq k\). Then \(F_m \leq F_k\), so \(F_m - F_n \leq F_k - 1 < F_k\).

Case 2: \(m = k+1\). For \(F_{k+1} - F_n > F_k\), we need \(F_n < F_{k+1} - F_k = F_{k-1}\). The largest such \(F_n\) is \(F_{k-2}\), and
  \[
  F_{k+1} - F_{k-2} = F_k + F_{k-3} > F_k + L,
  \]
  since \(F_{k-3} > L\). For any smaller \(F_n\), the difference is even larger.

Case 3: \(m \geq k+2\). The smallest difference exceeding \(F_k\) occurs when \(m = k+2\) and \(F_n = F_k\), yielding
  \[
  F_{k+2} - F_k = F_{k+1} = F_k + F_{k-1} > F_k + L,
  \]
  because \(F_{k-1} > F_{k-3} > L\). For larger \(m\), differences are larger still.
Therefore, \(F - F\) is not syndetic and so $|\mathbb{Z}:\{F_n\}|=\aleph_0$.
\end{proof}
Another thing that should be checked for  the Fibonacci numbers is its  $\infty$-difference length. With
do attention to the note after Lemma 1.1, we give a conjecture about it.\\
{\bf Conjecture II.} The $\infty$-difference length of Fibonacci numbers is infinite (i.e., $dl^\infty(F)=\infty$),
and so  $\mathbb{Z}=\langle F \rangle^{(\infty)}$.\\
\\
\textbf{Project II}(A big project for OEIS sequences of integers).\\
Thus far, one sees that the subindices and index stability (and possibly the infinite difference lengths) of very broad classes of common integer sequences can be determined using the obtained results, criteria, and Table~\ref{table2}. However, there are some OEIS sequences that should calculate their
subindices and $\infty$-difference lengths, and determine their index stability, separately.
\subsection{Sub-indices, index stability, and $\infty$-difference length of sets of rational and real numbers}
The following is a multi-part criterion for the index stability of rational sets (and sequences). According to this criterion, any non-index stable rational subset must be an unbounded, cardinally balanced generating set for $(\mathbb{Q}, +)$ whose difference set is syndetic.
\begin{theorem}[Rationals index stability Criteria]\label{rational criteria}
All the following conditions are necessary for $A$ to be non-index stable.\\
$($a$)$ $A$ and $A^c$ are infinite,\\
$($b$)$ $A$ is unbounded,\\
$($c$)$ $A$ is a generating but not difference-generating subset $($ i.e., $\langle A\rangle=\mathbb{Q}$ and $A-A\neq \mathbb{Q}$  $)$,\\
$($d$)$ $A-A$ is syndetic but not a subgroup,\\
$($e$)$ the equation $(A-A)\cap (X-X)=\{0\}$ has a finite maximal solution in $2^\mathbb{Q}$.\\
Hence, if a subset $A$ of rationales does not satisfy one of the conditions, then $A$ is index stable with indices $\aleph_0$ or 1.
\end{theorem}
\begin{proof}
The results are concluded from Theorem~\ref{InfiniteSubindexEquation},\ref{BasicTheorem}, Proposition~\ref{Countable groups}, and
Corollary~\ref{Cor Countable groups}.
\end{proof}
\begin{example}
 The set of all Bernoulli numbers is index stable in rationales with infinite index.
Because the Von Staudt-Clausen theorem \cite{vonStaudt1840} implies that the denominators of Bernoulli numbers \(B\) are square-free,
and rational numbers with square-free denominators make a proper subgroup of \(\mathbb{Q}\). Now, Corollary \ref{CountableCor}
implies that its index is \(\aleph_0\).
Also, we have  $|\mathbb{Q}:\{ \frac{1}{n}:n\in \mathbb{N}\}|=\aleph_0$ since it is bounded and so its difference set is not syndetic.
The set $\{ \frac{p}{q}:p,q\in \mathbb{P}\}$ dose not generate $\mathbb{Q}$
and so it is index stable with the index $\aleph_0$ (similarly for $\{ \frac{p}{q}:p,q\in \mathbb{P}\cup\{1\}\}$).
\end{example}
The next theorem and its proof are similar to Theorem~\ref{rational criteria}.
\begin{theorem}[Reals index stability Criteria]\label{real criteria}
All the following conditions are necessary for $A$ to be non index stable.\\
$($a$)$ $A$ and $A^c$ are uncountable,\\
$($b$)$ $A-A$ is not a subgroup,\\
$($d$)$ the equation $(A-A)+X=\mathbb{R}$ has a countable solution in $2^\mathbb{R}$,\\
$($e$)$ the equation $(A-A)\cap (X-X)=\{0\}$ has a countable maximal solution.\\
Hence, if a subset $A$ of real numbers does not satisfy one of the conditions, then $A$ is index stable $($ in $\mathbb{R}$)
with indices 1, $\aleph_0$, or $2^{\aleph_0}$,  if we accept the CH.
\end{theorem}

For every subgroup $\mathbb{F}\neq 0$ of $(\mathbb{R},+)$
(or $\mathbb{C}$),
we can say that if $A$ is a subset of $\mathbb{F}$ such that $A$ (equivalently $Dif(A)$) is bounded, then
 ($Dif(A)$ is not syndetic and so) $\aleph_0\leq|\mathbb{F}:A|^-\leq |\mathbb{F}:A|^+\leq c$.
The next theorem gives more related properties for the case $\mathbb{F}=\mathbb{R}$.
\begin{theorem}
Let $A$ be a subset of real numbers.
\begin{enumerate}[(a)]
\item If $A$ contains an interval, then $|\mathbb{R}:A|^+\leq \aleph_0$.
\item Every bounded subset containing an interval is index stable with index $\aleph_0$.
\item If $|\mathbb{R}:A|^+>\aleph_0$, then $A$ does not contain any interval.
\item There are infinitely many uncountable subsets $A$ of real numbers with $|\mathbb{R}:A|^+=c$.
\end{enumerate}
\end{theorem}
\begin{proof}
Note that this proof mentions the fact \cite{StackUncountableReals} in other words. If $A$ contains an interval, then $\Dif(A)$ contains a neighborhood of zero, and if
$B$ is uncountable, then it has an accumulation point and so contains a cauchy sequence $b_n$
such that all its elements are distinct. Therefore, $\Dif(A)\cap \Dif(B)$ is non-trivial and
so $B$ can not be a sub-factor related to $A$. Hence, $|\mathbb{R}:A|^+\leq \aleph_0$.
Also, $|\mathbb{R}:A|^-\geq \aleph_0$, since $A$ is bounded.
Thus, the parts (a) and (b) are proved, and part (c) is a direct result of (a).\\
Now, let $S$ be a basis for $\mathbb{R}$ as a vector space over $\mathbb{Q}$ and split
it as the union of two disjoint uncountable sets $T$ and $U$, put $B=\mbox{Span}\langle U\rangle$
and choose a subset $A$ of
the $\mathbb{Q}$-spans of $T$ such that $Dif(A)=\mbox{Span}\langle T\rangle$. Then,
$B\in \SubF(A)$ and $|B|=c$. Thus we arrive at (d).
\end{proof}
\begin{remark}
In view of the previous results, one find that the challenging subsets $A$ of real
numbers for studying sub-indices and index stability are included to one of the following classes:\\
(1) Bounded uncountable subsets that do not contain any interval. In this case, we have $\aleph_0\leq|\mathbb{R}:A|^-\leq |\mathbb{R}:A|^+\leq c$.
The Cantor type sets satisfy the conditions and so we arrive at the following question.
\begin{question}
Are the Cantor type sets index stable?
\end{question}
(2) Unbounded subsets containing an interval. this requires $1\leq|\mathbb{R}:A|^-\leq |\mathbb{R}:A|^+\leq \aleph_0$).
For this case, also we have a question.
\begin{question}
Does for every cardinal number $1\leq k \leq \aleph_0$ there exist an unbounded uncountable subset containing an interval with $|\mathbb{R}:A|=k$?
\end{question}
(3) Unbounded uncountable subsets that does not contain an interval. For example irrationals, transaudient numbers, etc.
\end{remark}
The last theorem implies that there are uncountable subsets of real
numbers with upper sub-indices $1$, $\aleph_0$ and $c$. Now,
regarding the fact that
$\mathbb{F}$ does not have any proper subgroup (and also sub-semigroup) with finite index,
we show that there are index stable subsets of $\mathbb{Q}$ and $\mathbb{R}$ with finite indices. Hence,
all nonzero subgroups of $\mathbb{R}$ are non-index stable, but they are index pervasive.
\begin{theorem}\label{pervasive}
The additive group of real (resp. complex) numbers and all its subgroups are index pervasive.
\end{theorem}
\begin{proof}
The identity subgroup is finite and index pervasive, clearly. Hence, let $H\neq \{0\}$ be a subgroup
of $(\mathbb{R},+)$ and we may assume $1\in H$ (because $H\cong \frac{1}{h}H$, for all $h\in H\neq \{0\}$). Given an integer $n>1$, put
$F=\{\pm1,\pm2,\dots,\pm (n-1)\}$, and if $n=1$, then put $F=\emptyset$. There is a subset $A$ of $H$ and a subset $A'$ of $\mathbb{Z}$ such that
$\mathcal{C}(H:A)=\mathcal{C}(\mathbb{Z}:A')=F$, by Proposition~\ref{codif}.
Thus
$$
\SubF^0(H:A)=\SubF^0(\mathbb{Z}:A')=\{[-x,y]_\mathbb{Z}: x,y\geq 0\; , \; x+y=n-1\},
$$
and so $|H:A|=n$, by Lemma~\ref{pervasive integers}. Now, we can complete the proof due to the fact that for every non-zero subgroup $H$
of $(\mathbb{R},+)$ and every cardinal number $\aleph_0\leq \kappa\leq |H|$, there is a subgroup of $H$ with
the index $\kappa$ in $H$.
\end{proof}

\section{Tables for subindices and index stability of subsets of a subgroup}
The following tables show that calculations of subindices of a subset $A\subseteq H\leq G$ depend on the finite and
infinite status of the index of $H$ in $G$. Then, for the case that $|G:H|$ is infinite, the values of subindices, and also
index stability of $A$ in $H$ and $G$, are dependent on the comparing status of $|G:H|$ with respect to $|H:A|^-$ and $|H:A|^+$.
We also design another useful (reduced) table for the case that $A$ is right index stable in $G$.
\begin{table}[h]
    \centering
    \renewcommand{\arraystretch}{1.5}
    \begin{tabular}{|l|l|}
        \hline
        \multicolumn{2}{|c|}{\textbf{Subindices cases for a subset $A\subseteq H\leq G$}} \\
        \hline
        \multirow{4}{*}{{\footnotesize $|G:H|$ is finite}} & $|G:A|^-$ is finite $\iff$ $|H:A|^-$ is finite , $|G:A|^+$ is finite $\iff$ $|H:A|^+$ is finite   \\
        \cline{2-2}
        &  $|H:A|^+ \mbox{ finite}\Rightarrow |G:A|^\pm\mbox{ finite},  1\leq \frac{|G:A|^-}{|G:H|}=|H:A|^-\leq |H:A|^+ = \frac{|G:A|^+}{|G:H|}\leq |H|$ \\
        \cline{2-2}
        & $A$ is right index stable in $G$ $\iff$ $A$ is right index stable in $H$\\
        \cline{2-2}
        & if $A$ is right index stable in $H$ or $G$, then (it is so in both, and) $|G:A|_r=|G:H||H:A|_r$  \\
        \hline
        \multirow{7}{*}{{\footnotesize $|G:H|$ is infinite}} & $|G:A|^\pm$ are infinite   \\
        \cline{2-2}
        &  if $|G:H|\leq |H:A|^-$, then $|G:A|^-=|H:A|^-$, $|G:A|^+=|H:A|^+$, and so $A$ is right index \\
        &  stable in $G$ if and only if $A$ is right index stable in $H$ (if it is the case, then $|G:A|_r=|H:A|_r$)\\
        \cline{2-2}
        &  if $|G:H|\geq |H:A|^+$, then $A$ is right index stable in $G$, and $|G:A|_r=|G:H|$\\
        \cline{2-2}
        & if $|H:A|^-< |G:H|< |H:A|^+$, then $|G:A|^+=|H:A|^+>|G:A|^-=|G:H|$. \\ & Thus, $A$ is not right index stable in $G$. \\
         \cline{2-2}
        & if $|H:A|^-=|G:H|< |H:A|^+$, then $|G:A|^+=|H:A|^+>|H:A|^-=|G:A|^-=|G:H|$. \\ & Thus, $A$ is not right index stable in both $G$ and $H$ \\
        \cline{2-2}
        & if $|H:A|^-< |G:H|=|H:A|^+$, then  $|G:A|^-=|H:A|^+=|G:H|=|G:A|^+$. \\ & Thus, $A$ is right index stable in $G$ ($|G:A|_r=|G:H|$) but not in $H$. \\
         \cline{2-2}
        & if $|H:A|^-=|G:H|=|H:A|^+$, then $A$ is right index stable in both $G$, $H$, \\ & and $|G:A|_r=|G:H|=|H:A|_r$  \\
        \hline
           \end{tabular}
   \caption{}
   \label{table2}
\end{table}

\begin{table}[h]
    \centering
    \renewcommand{\arraystretch}{1.5}
    \begin{tabular}{|l|l|}
        \hline
        \multicolumn{2}{|c|}{\textbf{Subindices cases for $A\subseteq H\leq G$ such that $A$ is right index stable in $G$}} \\
        \hline
       {\footnotesize $|G:H|$ is finite} &  $A$ is also index stable in $H$ and $|G:A|_r=|G:H||H:A|_r$   \\
        \hline
        \multirow{6}{*}{{\footnotesize $|G:H|$ is infinite}} & $\aleph_0\leq |G:H|\leq |G:A|_r=|G:H||H:A|^-=|G:H||H:A|^+$   \\
        \cline{2-2}
        &  $|H:A|^-\leq |G:H|$  $\Leftrightarrow$ $|G:A|_r=|G:H|$ $\Leftrightarrow$  $|H:A|^+\leq |G:H|$ \\
        \cline{2-2}
        & $|H:A|^-\geq |G:H|$ $\Leftrightarrow$ $A$ is index stable in $H$, $|G:A|_r=|H:A|_r$\\
        \cline{2-2}
        & $|H:A|^+\geq |G:H|$ $\Leftrightarrow$ $|H:A|^+=|G:A|_r=|G:H||H:A|^-$  \\
         \cline{2-2}
        & $|H:A|^+\geq |G:H|\geq |H:A|^-$ $\Leftrightarrow$ $|H:A|^+=|G:A|_r=|G:H|\geq |H:A|^-$  \\
        \cline{2-2}
        &  $|H:A|^->|G:H|$  $\Leftrightarrow$  $|H:A|^+> |G:H|$ $\Leftrightarrow$ $|H:A|^-=|H:A|^+>|G:H|$ \\
        \hline
           \end{tabular}
 \caption{}
  \label{table3}
\end{table}
For the case that $A$ is right index stable in $H$, we don't need another table, since it obtains
all the case results of Table\ref{table3} together with the condition  $|H:A|^+=|H:A|^-$.

\newpage

\bibliographystyle{amsplain}

\end{document}